\documentclass[letterpaper,10pt,twocolumn]{IEEEtran}

\title{Cyber-Physical Systems under Attack -- Part I:\\ Models  and
  Fundamental Limitations}

\title{Attack Detection and Identification in Cyber-Physical Systems
  -- Part I:\\ Models and Fundamental Limitations}

\author{Fabio Pasqualetti, Florian D\"orfler, and Francesco Bullo
  \thanks{This material is based upon work supported in part by NSF grant
    CNS-1135819 and by the Institute for Collaborative Biotechnologies
    through grant W911NF-09-0001 from the U.S. Army Research Office.}
  \thanks{Fabio Pasqualetti, Florian D\"orfler, and Francesco Bullo are
    with the Center for Control, Dynamical Systems and Computation,
    University of California at Santa Barbara, {\tt
      \{fabiopas,dorfler,bullo\}@engineering.ucsb.edu}}%
}

\usepackage{amsmath}
\usepackage{amssymb}
\usepackage{mathrsfs}
\usepackage{graphicx}
\usepackage{cite}

\newtheorem{theorem}{Theorem}[section]

\newtheorem{lemma}[theorem]{Lemma}
\newtheorem{definition}{Definition}
\newtheorem{remark}{Remark}
\newtheorem{example}{Example}

\usepackage[noend]{algorithmic}
\usepackage{amsmath,graphicx,epsfig,color,amsfonts,algorithm,subfigure}

\newcommand\oprocendsymbol{\hbox{$\square$}}
\newcommand\oprocend{\relax\ifmmode\else\unskip\hfill\fi\oprocendsymbol}

\newenvironment{pfof}[1]{\vspace{1ex}\noindent{\itshape Proof of
    #1:}\hspace{0.5em}} {\hfill\QED\vspace{1ex}}

\newenvironment{pf}{\vspace{1ex}\noindent{\itshape
    Proof:}\hspace{0.5em}} {\hfill\QED\vspace{1ex}}

\renewcommand{\theenumi}{(\roman{enumi})}

\newcommand{\setdef}[2]{\{#1 \; : \; #2\}}
\newcommand{\subscr}[2]{{#1}_{\textup{#2}}}
\newcommand{\supscr}[2]{{#1}^{\textup{#2}}}
\newcommand{\until}[1]{\{1,\dots,#1\}}

\usepackage{color}

\newcommand{\Ker}{\operatorname{Ker}}

\newcommand{\Image}{\operatorname{Im}}

\newcommand{\V}{\mathcal{V}}

\newcommand{\real}{\mathbb{R}}
\newcommand{\complex}{\mathbb{C}}
\newcommand{\transpose}{\mathsf{T}} 

\newcommand{\mc}{\mathcal}

\begin{document}
\maketitle

\begin{abstract}
  Cyber-physical systems integrate computation, communication, and
  physical capabilities to interact with the physical world and
  humans. Besides failures of components, cyber-physical systems are
  prone to malignant attacks, and specific analysis tools as well as
  monitoring mechanisms need to be developed to enforce system
  security and reliability. This paper proposes a unified framework to
  analyze the resilience of cyber-physical systems against attacks
  cast by an omniscient adversary. We model cyber-physical systems as
  linear descriptor systems, and attacks as exogenous unknown
  inputs. Despite its simplicity, our model captures various
  real-world cyber-physical systems, and it includes and generalizes
  many prototypical attacks, including stealth, (dynamic) false-data
  injection and replay attacks. First, we characterize fundamental
  limitations of static, dynamic, and active monitors for attack
  detection and identification. Second, we provide constructive
  algebraic conditions to cast undetectable and unidentifiable
  attacks. Third, by using the system interconnection structure, we
  describe graph-theoretic conditions for the existence of
  undetectable and unidentifiable attacks. Finally, we validate our
  findings through some illustrative examples with different
  cyber-physical systems, such as a municipal water supply network and
  two electrical power grids.
\end{abstract}

\section{Introduction}

\emph{Cyber-physical systems} arise from the tight integration of
physical processes, computational resources, and communication
capabilities. More precisely, processing units monitor and control
physical processes by means of sensors and actuators
networks. Examples of cyber-physical systems include transportation
networks, power generation and distribution networks, water and gas
distribution networks, and advanced communication systems. Due to the
crucial role of cyber-physical systems in everyday life,
cyber-physical security needs to be promptly addressed.

Besides failures and attacks on the physical infrastructure,
cyber-physical systems are also prone to cyber attacks on their data
management and communication layer. Recent studies and real-world
incidents have demonstrated the inability of existing security methods
to ensure a safe and reliable functionality of cyber-physical
infrastructures against unforeseen failures and, possibly, external
attacks \cite{ARM-RLE:10,SS-AH-MG:12,JS-MM:07,AAC-SA-SS:08}. The protection
of critical infrastructures is, as of today, one of the main focus of
the Department of Homeland Security\,\cite{GEA-DML:05}.

Concerns about security of control systems are not new, as the
numerous manuscripts on systems fault detection, isolation, and
recovery testify; see for example
\cite{MB-IVN:93,SXD:08}. Cyber-physical systems, however, suffer from
specific vulnerabilities which do not affect classical control
systems, and for which appropriate detection and identification
techniques need to be developed. For instance, the reliance on
communication networks and standard communication protocols to
transmit measurements and control packets increases the possibility of
intentional and worst-case (cyber) attacks against physical plants. On
the other hand, information security methods, such as authentication,
access control, message integrity, and cryptography methods, appear
inadequate for a satisfactory protection of cyber-physical systems.
Indeed, these security methods do not exploit the compatibility of the
measurements with the underlying physical process and control
mechanism, which are the ultimate objective of a protection scheme
\cite{AAC-SA-BS-AG-AP-SSS:09}. Moreover, such information security
methods are not effective against insider attacks carried out by
authorized entities, as in the famous Maroochy Water Breach case
\cite{JS-MM:07}, and they also fail against attacks targeting directly
the physical dynamics \cite{CLD-JVS-FA:96}.

\noindent
\textbf{Related work.} The analysis of vulnerabilities of
cyber-physical systems to external attacks has received increasing
attention in the last years. The general approach has been to study
the effect of specific attacks against particular systems. For
instance, in \cite{SA-AC-SS:09} \emph{deception} and \emph{denial of
  service} attacks against a networked control system are introduced,
and, for the latter ones, a countermeasure based on semi-definite
programming is proposed. Deception attacks refer to the possibility of
compromising the integrity of control packets or measurements, and
they are cast by altering the behavior of sensors and
actuators. Denial of service attacks, instead, compromise the
availability of resources by, for instance, jamming the communication
channel. In \cite{YL-MKR-PN:09} \emph{false data} injection attacks
against static state estimators are introduced. False data injection
attacks are specific deception attacks in the context of static
estimators. It is shown that undetectable false data injection attacks
can be designed even when the attacker has limited resources. In a
similar fashion, \emph{stealthy deception attacks} against the
Supervisory Control and Data Acquisition system are studied, among
others, in \cite{AT-AS-HS-KHJ-SSS:10,SA-XL-SS-AMB:10}. In
\cite{YM-BS:10a} the effect of \emph{replay attacks} on a control
system is discussed. Replay attacks are cast by hijacking the sensors,
recording the readings for a certain amount of time, and repeating
such readings while injecting an exogenous signal into the system. It
is shown that this type of attack can be detected by injecting a
signal unknown to the attacker into the system. In \cite{RS:11} the
effect of \emph{covert attacks} against networked control systems is
investigated. Specifically, a parameterized decoupling structure
allows a covert agent to alter the behavior of the physical plant
while remaining undetected from the original controller. In
\cite{MZ-SM:11} a resilient control problem is studied, in which
control packets transmitted over a network are corrupted by a human
adversary. A receding-horizon Stackelberg control law is proposed to
stabilize the control system despite the attack. Recently the problem
of estimating the state of a linear system with corrupted measurements
has been studied \cite{FH-PT-SD:11}. More precisely, the maximum
number of faulty sensors that can be tolerated is characterized, and a
decoding algorithm is proposed to detect corrupted
measurements. Finally, security issues of some specific cyber-physical
systems have received considerable attention, such as power networks
\cite{ARM-RLE:10,SS-AH-MG:12,AT-AS-HS-KHJ-SSS:10,DG-HS:10,ES:04,FP-AB-FB:10u,FP-FD-FB:11i,AHMR-ALG:11,CLD-JVS-FA:96},
linear networks with misbehaving components
\cite{SS-CH:10a,FP-AB-FB:09b}, and water networks
\cite{JS-MM:07,SA-XL-SS-AMB:10,RS:11,DGE-MMP:10}.

\noindent
\textbf{Contributions.} The contributions of this paper are as
follows. First, we describe a unified modeling framework for
cyber-physical systems and attacks. Motivated by existing
cyber-physical systems and proposed attack scenarios, we model a
cyber-physical system under attack as a descriptor system subject to
unknown inputs affecting the state and the measurements. For our
model, we define the notions of {\em detectability} and {\em
  identifiability} of an attack by its effect on output measurements.
Informed by the classic work on geometric control
theory~\cite{WMW:85}, our framework includes the \emph{deterministic
  static detection problem} considered in
\cite{YL-MKR-PN:09,AT-AS-HS-KHJ-SSS:10}, and the prototypical
deception and denial of service \cite{SA-AC-SS:09}, stealth
\cite{DG-HS:10}, (dynamic) false-data injection \cite{YM-BS:10b},
replay \cite{YM-BS:10a}, and covert attacks \cite{RS:11} as special
cases. Second, we show the fundamental limitations of static, dynamic,
and active detection and identification procedures. Specifically, we
show that static detection procedures are unable to detect any attack
affecting the dynamics, and that attacks corrupting the measurements
can be easily designed to be undetectable. On the contrary, we show
that undetectability in a dynamic setting is much harder to achieve
for an attacker. Specifically, a cyber-physical attack is undetectable
if and only if the attackers' signal excites uniquely the zero
dynamics of the input/output system. Additionally, we show that active
monitors capable of injecting test signals are as powerful as dynamic
(passive) monitors, since an attacker can design undetectable and
unidentifiable attacks without knowing the signal injected by the
monitor into the system. This analysis bring us also to the conclusion
that undetectable attacks can be cast even without knowledge of system
noise. Third, we provide a graph theoretic characterization of
undetectable attacks. Specifically, we borrow some tools from the
theory of structured systems, and we identify conditions on the system
interconnection structure for the existence of undetectable
attacks. These conditions are \emph{generic}, in the sense that they
hold for almost all numerical systems with the same structure, and
they can be efficiently verified. As a complementary result, we extend
a result of \cite{JWW:91} on structural left-invertibility to regular
descriptor systems. Fourth and finally, we illustrate the potential
impact of our theoretical findings through compelling examples. In
particular, we design (i) an undetectable state attack to destabilize
the WSSC 3-machine 6-bus power system, (ii) an undetectable output
attack for the IEEE 14 bus system, and (iii) an undetectable state and
output attack to steal water from a reservoir of the EPANET network
model 3. Through these examples we show the advantages of dynamic
monitors against static ones, and we provide insight on the design of
attacks.


\noindent
\textbf{Paper organization.} The remainder of the paper is organized
as follows. Section \ref{sec:example_systems} presents some examples
of cyber-physical systems. Section \ref{sec:setup} contains our models
of cyber-physical systems, attacks, and monitors.
Our main results are presented in Section \ref{sec:static_dynamic} and
in Section \ref{sec:graph_conditions}. In particular, in Section
\ref{sec:static_dynamic} we describe the fundamental limitations of
static, dynamic, and active detectors, and we provide constructive
algebraic conditions for the existence of undetectable and
unidentifiable attacks. In Section \ref{sec:graph_conditions},
instead, we derive graph-theoretic conditions for the existence of
undetectable and unidentifiable attacks. Finally, Section
\ref{sec:example} and Section \ref{sec:conclusion} contain,
respectively, our illustrative examples and our conclusion.


\section{Examples of cyber-physical
  systems}\label{sec:example_systems}
We now motivate our study by introducing important cyber-physical
systems requiring advanced security mechanisms.

\subsection{Power networks}\label{example:power}
Future power grids will combine physical dynamics with a sophisticated
coordination infrastructure. The cyber-physical security of the grid
has been identified as an issue of primary concern
\cite{ARM-RLE:10,SS-AH-MG:12}, which has recently attracted the interest of
the control and power systems communities, see
\cite{AT-AS-HS-KHJ-SSS:10,DG-HS:10,ES:04,FP-AB-FB:10u,FP-FD-FB:11i,AHMR-ALG:11}.

We adopt the small-signal version of the classical
structure-preserving power network model; see
\cite{ES:04,FP-AB-FB:10u} for a detailed derivation from the full
nonlinear structure-preserving power network model. Consider a
connected power network consisting of $n$ generators
$\{g_{1},\dots,g_{n}\}$ and $m$ load buses
$\{b_{n+1},\dots,b_{n+m}\}$. The interconnection structure of the
power network is encoded by a connected susceptance-weighted
graph. The generators $g_{i}$ and buses $b_{i}$ are the vertex set of
this graph, and the edges are the transmission lines $\{b_{i},b_{j}\}$
weighted by the susceptance between buses $b_{i}$ and $b_{j}$, as well
as the connections $\{g_{i},b_{i}\}$ weighted by the transient
susceptance between generator $g_{i}$ and its adjacent bus
$b_{i}$. The Laplacian associated with the susceptance-weighted graph
is the symmetric susceptance matrix $\mc L = \left[
\begin{smallmatrix}
  \subscr{\mc L}{gg} & \subscr{\mc L}{gl}\\
  \subscr{\mc L}{lg} & \subscr{\mc L}{ll}
\end{smallmatrix}
\right] \in \mathbb R^{(n+m) \times (n+m)}$, where the first $n$ rows
are associated with the generators and the last $m$ rows correspond to
the buses. The dynamic model of the power network is
\begin{align}
  \label{eq: power network descriptor system model} 
  	\begin{bmatrix}
         I\!\!  & 0\!\!  & 0\\
         0\!\! & \subscr{M}{g}\!\! & 0\\
         0\!\! & 0\!\!  & 0    
  \end{bmatrix}
  \begin{bmatrix}
  \dot \delta(t) \\ \dot \omega(t) \\ \dot\theta(t)
  \end{bmatrix}
  \!=\!
  -
  \begin{bmatrix}
         0\!\! & -I\!\! & 0\\
         \subscr{\mc L}{gg}\!\! & \subscr{D}{g}\!\! & \subscr{\mc L}{gl}\\
         \subscr{\mc L}{lg}\!\! & 0\!\! & \subscr{\mc L}{ll}
  \end{bmatrix}
  \!\!
      \begin{bmatrix}
	\delta(t) \\ \omega(t) \\ \theta(t)
  \end{bmatrix}
   \!+\!
    \begin{bmatrix}
    0 \\ P_{\omega}(t) \\ P_{\theta}(t)
    \end{bmatrix}
    \!, 
\end{align} 
where $\delta(t) \in \mathbb R^{n}$ and $\omega(t) \in \mathbb R^{n}$
denote the generator rotor angles and frequencies, and $\theta(t) \in
\mathbb R^{m}$ are the voltage angles at the buses. The terms
$\subscr{M}{g}$ and $D_\textup{g}$ are the diagonal matrices of the
generator inertial and damping coefficients, and the inputs
$P_{\omega}(t)$ and $P_{\theta}(t)$ are due to {\em known} changes in
mechanical input power to the generators or real power demand at the
loads.

\subsection{Mass transport networks}\label{example:water}
Mass transport networks are prototypical examples of cyber-physical
systems modeled by differential-algebraic equations, such as gas
transmission and distribution networks \cite{AO:87}, large-scale
process engineering plants \cite{AK-PD:99}, and water
networks. Examples of water networks include open channel flows
\cite{XL-VF:09} for irrigation purposes and municipal water networks
\cite{JB-BG-MCS:09,PFB-KEL-BWK:06}. The vulnerability of open channel
networks to cyber-physical attacks has been studied in
\cite{SA-XL-SS-AMB:10,RS:11}, and municipal water networks are also
known to be susceptible to attacks on the hydraulics \cite{JS-MM:07}
and biochemical contamination threats \cite{DGE-MMP:10}.

We focus on the hydraulics of a municipal water distribution network,
as modeled in \cite{JB-BG-MCS:09,PFB-KEL-BWK:06}. The water network
can be modeled as a directed graph with node set consisting of
reservoirs, junctions, and storage tanks, and with edge set given by
pipes, pumps, and valves that are used to convey water from source
points to consumers. The key variables are the pressure head $h_{i}$
at each node $i$ in the network as well as the flows $Q_{ij}$ from
node $i$ to $j$. The hydraulic model governing the network dynamics
includes constant reservoir heads, flow balance equations at junctions
and tanks, and pressure difference equations along all edges:
\begin{align}
  \begin{split}
    \mbox{reservoir } i:&\quad
    h_{i} = \supscr{h_{i}}{reservoir} = \text{constant}\,,
    \\
    \mbox{junction } i:&\quad
    d_{i} =  \sum\nolimits_{j \to i} \!Q_{ji} - \sum\nolimits_{i \to k} \!Q_{ik} \,, 
    \\
    \mbox{tank } i:&\quad
    A_{i} \dot h_{i} =  \sum\nolimits_{j \to i} \!Q_{ji} - \sum\nolimits_{i \to k} \!Q_{ik}\,, 
    \\
    \mbox{pipe } (i,j):&\quad
    Q_{ij} = Q_{ij}(h_{i} - h_{j}) \,,
    \\
    \mbox{pump } (i,j):&\quad
    h_{j} - h_{i} = + \supscr{\Delta h_{ij}}{pump} = \text{constant}\,,
    \\
    \mbox{valve } (i,j):&\quad
    h_{j} - h_{i} = - \supscr{\Delta h_{ij}}{valve} = \text{constant}\,.
  \end{split}
  \label{eq:water_network_model}
\end{align}
Here $d_{i}$ is the demand at junction $i$, $A_{i}$ is the (constant)
cross-sectional area of storage tank $i$, and the notation ``$j \to
i$'' denotes the set of nodes $j$ connected to node $i$. The flow
$Q_{ij}$ depends on the pressure drop $h_{i} - h_{j}$ along pipe
according to the Hazen-Williams equation $Q_{ij}(h_{i} - h_{j}) =
g_{ij} |h_{i} - h_{j}|^{1/1.85-1} \cdot (h_{i} - h_{j})$, where
$g_{ij}>0$ is the pipe conductance.


Other interesting examples of cyber-physical systems captured by our
modeling framework are sensor networks, dynamic Leontief models of
multi-sector economies, mixed gas-power energy networks, and
large-scale control systems.

\section{Mathematical Modeling Of Cyber-physical Systems, Monitors,
  and Attacks}\label{sec:setup}
In this section we model cyber-physical systems under attack as linear
time-invariant descriptor systems subject to unknown inputs. This
modeling framework is very general and includes most of the existing
cyber-physical models, attacks, and fault scenarios. Indeed, as shown
in Section \ref{sec:example_systems}, many interesting real-world
cyber-physical systems contain conserved physical quantities leading
to differential-algebraic system descriptions, and, as we show later,
most attack and fault scenarios can be modeled by additive inputs
affecting the state and the measurements.


\smallskip
\noindent \textbf{Model of cyber-physical systems under attack.}
We consider the linear time-invariant descriptor system%
\footnote{The results stated in this paper for continuous-time
  descriptor systems hold also for discrete-time descriptor systems
  and nonsingular systems. Moreover, we neglect the presence of known
  inputs, since, due to the linearity of system \eqref{eq:
    cyber_physical_fault}, they do not affect our results on the
  detectability and identifiability of unknown input attacks.}
\begin{align}\label{eq: cyber_physical_fault}
  \begin{split}
    E \dot x(t) &= Ax(t) + Bu(t),\\
    y(t) &= C x(t) + Du(t),
  \end{split}
\end{align}
where $x(t) \in \real^n$, $y(t) \in \real^p$, $E \in \real^{n \times
  n}$, $A \in \real^{n \times n}$, $B \in \real^{n \times m}$, $C \in
\real^{p \times n}$, and $D \in \real^{p \times m}$. Here the matrix
$E$ is possibly singular, and the input terms $Bu(t)$ and $Du(t)$ are
unknown signals describing disturbances affecting the plant. Besides
reflecting the genuine failure of systems components, these
disturbances model the effect of an attack against the cyber-physical
system (see below for our attack model). For notational convenience
and without affecting generality, we assume that each state and output
variable can be independently compromised by an attacker. Thus, we let
$B = \begin{bmatrix}I , 0\end{bmatrix}$ and $D = \begin{bmatrix}0 ,
  I\end{bmatrix}$ be partitioned into identity and zero matrices of
appropriate dimensions, and, accordingly, $u(t) = \begin{bmatrix}
  u_{x}(t)^{\transpose},u_{y}(t)^{\transpose} \end{bmatrix}^{\transpose}$.
Hence, the \emph{attack} $(Bu(t),Du(t)) = (u_{x}(t),u_{y}(t))$ can be
classified as \emph{state attack} affecting the system dynamics and as
\emph{output attack} corrupting directly the measurements vector.

The attack signal $t \mapsto u(t) \in \mathbb R^{n+p}$
depends upon the specific attack strategy. In the presence of $k \in
\mathbb{N}_0$, $k \le n+p$, attackers indexed by the {\em attack set}
$K \subseteq \until{n+p}$ only and all the entries $K$ of $u(t)$ are
nonzero over time. To underline this sparsity relation, we sometimes
use $u_K(t)$ to denote the {\em attack mode}, that is the subvector of
$u(t)$ indexed by $K$. Accordingly, the pair $(B_{K},D_{K})$, where
$B_K$ and $D_K$ are the submatrices of $B$ and $D$ with columns in
$K$, to denote the {\em attack signature}. Hence, $Bu(t) = B_K u_K
(t)$, and $Du(t) = D_K u_K (t)$. Since the matrix $E$ may be singular,
we make the following assumptions on system \eqref{eq:
  cyber_physical_fault}:
\begin{enumerate}\setcounter{enumi}{2}
\item[(A1)] the pair $(E,A)$ is regular, that is, $\textup{det}(sE - A)$ does
  not vanish identically,
\item[(A2)] the initial condition $x(0) \in \mathbb R^{n}$ is
  consistent, that is, $(Ax(0) + B u(0)) \perp \Ker(E^{\transpose}) = 0$; and
\item[(A3)] the input signal $u(t)$ is smooth.
\end{enumerate}
The regularity assumption (A1) assures the existence of a unique
solution $x(t)$ to \eqref{eq: cyber_physical_fault}. Assumptions (A2)
and (A3) simplify the technical presentation in this paper since they
guarantee smoothness of the state trajectory $x(t)$ and the
measurements $y(t)$; see \cite[Lemma 2.5]{TG:93} for further
details. 
The degree of smoothness in assumption (A3) depends on the index of
$(E,A)$, see \cite[Theorem 2.42]{PK-VLM:06}, and continuity of $u(t)$
is sufficient for the index-one examples presented in Section \ref{sec:example_systems}. In Section
\ref{Subsection: smoothness issues} we discuss the results in this
paper if assumptions (A2) and (A3) are dropped.

\smallskip
\noindent \textbf{Model of static, dynamic, and active monitors.}
A \emph{monitor} is a pair $(\Phi, \gamma(t))$, where $\Phi \, :\,
\Lambda \,\rightarrow \, \Psi$ is an algorithm, and $\gamma \,:\,
\real \mapsto \mathbb R^{n+p}$ is a signal. In particular, $\Lambda$
is the algorithm input to be specified later, $\Psi =
\{\psi_1,\psi_2\}$, with $\psi_1 \in \{\textup{True},\textup{False}\}$
and $\psi_2 \subseteq \until{n+p}$, is the algorithm output, and $(B
\gamma (t),D \gamma (t))$ is an auxiliary input injected
by the monitor into the system \eqref{eq: cyber_physical_fault}. In
this work we consider the following classes of monitors for the system
\eqref{eq: cyber_physical_fault}.

\begin{definition}{\bf \emph{(Static monitor)}}\label{static_monitor}
  A \emph{static monitor} is a monitor with $\gamma (t) = 0 \; \forall
  t \in \mathbb{R}_{\ge 0}$, and $\Lambda = \{C,y(t) \; \forall t \in
  \mathbb{N}\}$.
\end{definition}

Note that static monitors do not exploit relations among measurements
taken at different time instants. An example of static monitor is the
\emph{bad data detector} \cite{AA-AGS:04}.

\begin{definition}{\bf \emph{(Dynamic monitor)}}\label{dynamic_monitor}
  A \emph{dynamic monitor} is a monitor with $\gamma (t) = 0 \;
  \forall t \in \mathbb{R}_{\ge 0}$, and $\Lambda = \{E,A,C,y(t)
  \;\forall t \in \mathbb{R}_{\ge 0}\}$.
\end{definition}
Differently from static monitors, dynamic monitors have knowledge of
the system dynamics generating $y(t)$ and may exploit temporal
relations among different measurements. The filters defined in
\cite{FP-FD-FB:11i} are examples of dynamic monitors.

\begin{definition}{\bf \emph{(Active monitor)}}\label{dynamic_monitor}
  An \emph{active monitor} is a monitor with $\gamma (t) \neq 0$ for
  some $t \in \mathbb{R}_{\ge 0}$, and $\Lambda = \{E,A,C,y(t)
  \;\forall t \in \mathbb{R}_{\ge 0}\}$.
\end{definition}
Active monitors are dynamic monitors with the ability of modifying the
system dynamics through an input. An example of active monitor is
presented in \cite{YM-BS:10a} to detect replay attacks. The objective
of a monitor is twofold:

\begin{definition}{\bf \emph{(Attack
      detection)}}\label{attack_detection}
  A nonzero attack $(B_K u_K (t), D_K u_K(t))$ is detected by a monitor if $\psi_1 =
  \textup{True}$.
\end{definition}

\begin{definition}{\bf \emph{(Attack identification)}}\label{attack_identification}
  A nonzero attack $(B_K u_K (t), D_K u_K(t))$ is identified by a
  monitor if $\psi_2 = K$.
\end{definition}

An attack is called \emph{undetectable} (respectively
\emph{unidentifiable}) by a monitor if it fails to be detected
(respectively identified) by every monitor in the same class. Of
course, an undetectable attack is also unidentifiable, since it cannot
be distinguished from the zero attack. By extension, an attack set $K$
is undetectable (respectively unidentifiable) if there exists an
undetectable (respectively unidentifiable) attack $(B_K u_K,D_K u_K)$.

\smallskip
\noindent \textbf{Model of attacks.}
In this work we consider colluding omniscient attackers with the
ability of altering the cyber-physical dynamics through exogenous
inputs. In particular we let the attack $(B u(t),Du(t))$ in \eqref{eq:
  cyber_physical_fault} be designed based on knowledge of the system
structure and parameters $E, A, C$, and the full state $x(t)$ at all
times. Additionally, attackers have unlimited computation
capabilities, and their objective is to disrupt the physical state or
the measurements while avoiding detection.

\smallskip
\begin{remark}{\bf \emph{(Existing attack strategies as subcases)}}
  \begin{figure}[tb]
    \centering \subfigure[Static stealth attack]{
      \includegraphics[width=.485\columnwidth]{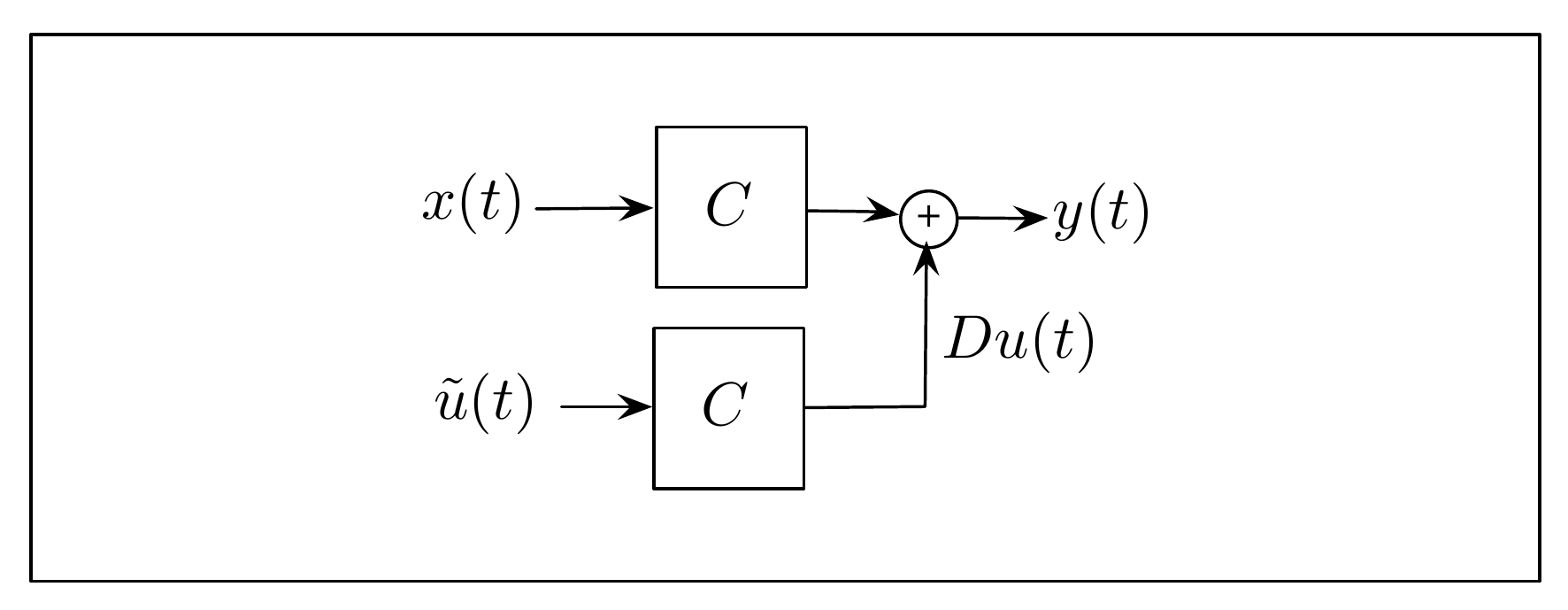}
      \label{stealth_attack}
    } \!\!\!\!\!\!\subfigure[Replay attack]{
      \includegraphics[width=.485\columnwidth]{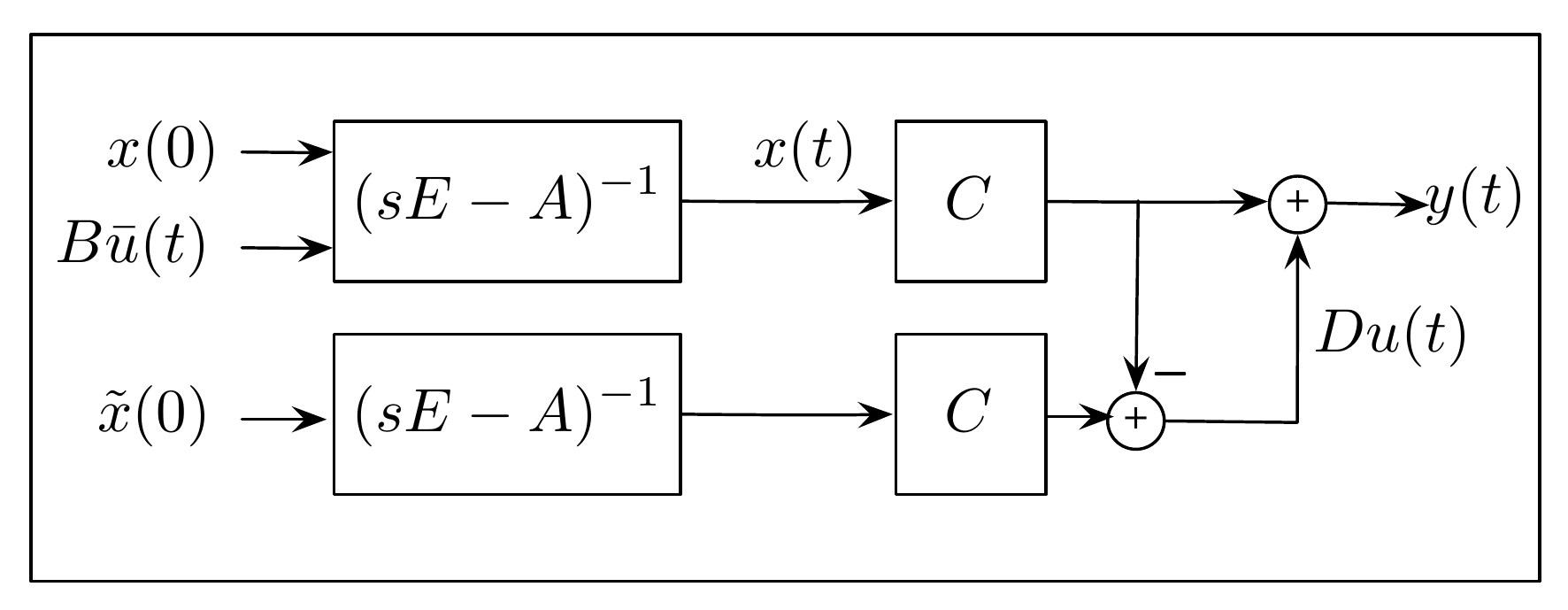}
      \label{replay_attack}
    }\\
    \subfigure[Covert attack]{
   \includegraphics[width=.485\columnwidth]{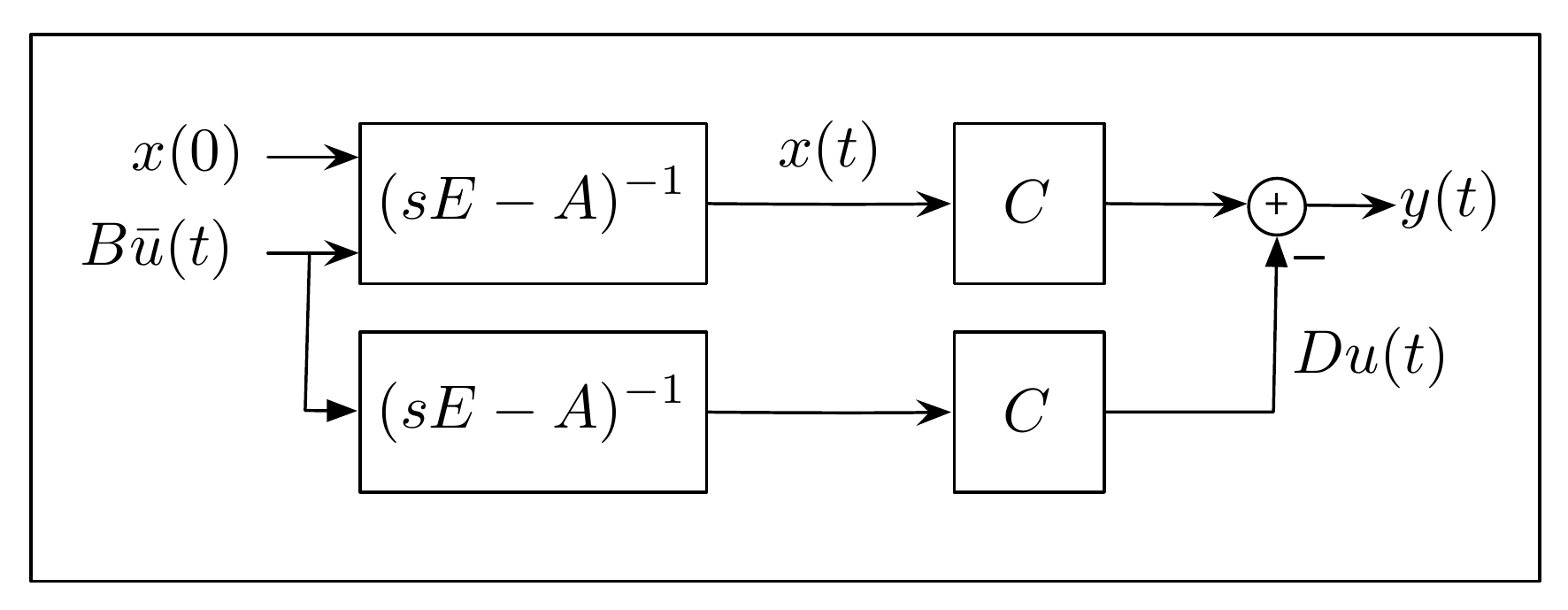}
      \label{covert_attack}
    } \!\!\!\!\!\!\subfigure[Dynamic false data injection]{
      \includegraphics[width=.485\columnwidth]{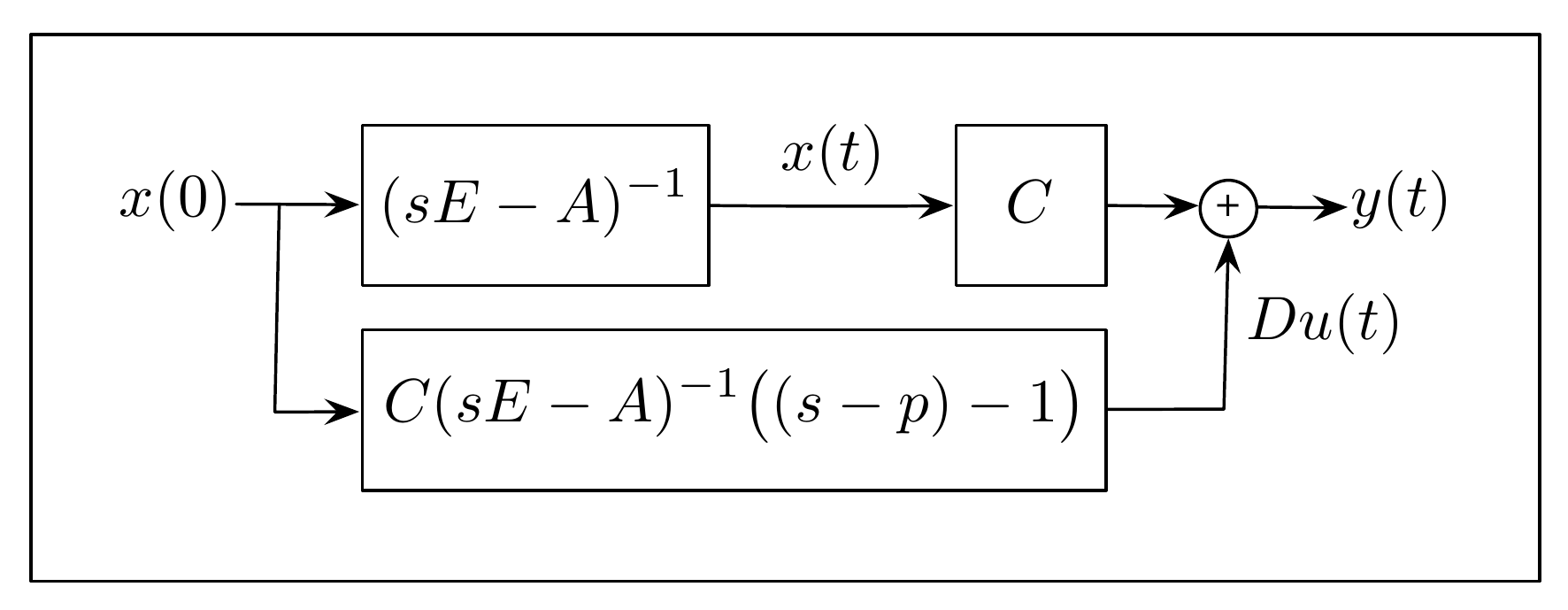}
      \label{dynamic_attack}
    } 
    \caption[Optional caption for list of figures]{ A block diagram
      illustration of prototypical attacks is here reported. In
      Fig. \ref{stealth_attack} the attacker corrupts the measurements
      $y(t)$ with the signal $D_K u_K (t) \in \Image (C)$. Notice that
      in this attack the dynamics of the system are not considered. In
      Fig. \ref{stealth_attack} the attacker affects the output so
      that $y(t) = y(x(0),[\bar u_K^\transpose \;
      u_K^\transpose]^\transpose,t) = y(\tilde x(0),0,t)$. The covert
      attack in Fig. \ref{covert_attack} is a feedback version of the
      replay attack, and it can be explained analogously. In
      Fig. \ref{dynamic_attack} the attack is such that the unstable
      pole $p$ is made unobservable.}
    \label{fig:prototipical_attacks}
  \end{figure}
  The following prototypical attacks can be modeled and analyzed
  through our theoretical framework:
  \begin{enumerate}
  \item \emph{stealth attacks} defined in \cite{DG-HS:10} correspond
    to output attacks compatible with the measurements equation;
  \item \emph{replay attacks} defined in \cite{YM-BS:10a} are state
    and output attacks which affect the system dynamics and reset the
    measurements;
  \item \emph{covert attacks} defined in \cite{RS:11} are closed-loop
    replay attacks, where the output attack is chosen to cancel out
    the effect on the measurements of the state attack; and
  \item \emph{(dynamic) false-data injection attacks} defined in
    \cite{YM-BS:10b} are output attacks rendering an unstable mode (if
    any) of the system unobservable.
  \end{enumerate}
  A possible implementation of the above attacks in our model is
  illustrated in Fig. \ref{fig:prototipical_attacks}.  \oprocend
\end{remark}

To conclude this section we remark that the examples presented in
Section \ref{sec:example_systems} are captured in our framework. In
particular, classical power networks failures modeled by additive
inputs include sudden change in the mechanical power input to
generators, lines outage, and sensors failure; see \cite{FP-FD-FB:11i}
for a detailed discussion. Analogously, for a water network, faults
modeled by additive inputs include leakages, variation in demand,
and failures of pumps and sensors. Possible cyber-physical attacks in both power and water networks 
include comprising measurements \cite{SA-XL-SS-AMB:10,AT-AS-HS-KHJ-SSS:10,YL-MKR-PN:09} 
and attacks on the control architecture or the physical state itself \cite{SS-AH-MG:12,CLD-JVS-FA:96,AHMR-ALG:11,JS-MM:07}.







\section{Limitations of static, dynamic and active monitors for
  detection and identification}\label{sec:static_dynamic}
The objective of this section is to highlight fundamental detection
and identification limitations of static, dynamic, and active
monitors. In particular, we show that the performance of widely used
static monitors can be greatly improved by exploiting the system
dynamics. On the other hand, the possibility of injecting monitoring
signals does not improve the detection capabilities of a (passive)
dynamic monitor.

Observe that a cyber-physical attack is undetectable if there exists a
normal operating condition of the system under which the output would
be the same as under the perturbation due to the attacker. Let
$y(x_0,u,t)$ be the output sequence generated from the initial state
$x_0$ under the attack signal $u(t)$.

\begin{lemma}{\bf \emph{(Undetectable
      attack)}}\label{undetectable_input}
  For the linear descriptor system \eqref{eq: cyber_physical_fault},
  the attack $(B_K u_K,D_K u_K)$ is undetectable by a static monitor
  if and only if $y(x_1,u_K,t) = y(x_2,0,t)$ for some initial
  condition $x_1, x_2 \in \real^n$ and for $t \in \mathbb{N}_0$. If
  the same holds for $t \in \real_{\ge 0}$, then the attack is also
  undetectable by a dynamic monitor.
\end{lemma}


Lemma \ref{undetectable_input} follows from the fact that our monitors
are deterministic, so that $y(x_1,u_K,t)$ and $y(x_2,0,t)$ lead to the
same output $\psi_1$. A more general concern than detectability is
identifiability of attackers, that is, the possibility to distinguish
from measurements between the action of two distinct attacks. We
quantify the strength of an attack through the cardinality of the
attack set. Since an attacker can independently compromise any state
variable or measurement, every subset of the states and measurements
of fixed cardinality is a possible attack set.

\begin{lemma}{\bf \emph{(Unidentifiable
      attack)}}\label{unidentifiable_input}
  For the linear descriptor system \eqref{eq: cyber_physical_fault},
  the attack $(B_K u_K,D_K u_K)$ is unidentifiable by a static monitor
  if and only if $y(x_1,u_K,t) = y(x_2,u_R,t)$ for some initial
  condition $x_1, x_2 \in \real^n$, attack $(B_R u_R,D_R u_R)$ with
  $|R|\le|K|$ and $R\neq{K}$, and for $t \in \mathbb{N}_0$. If the
  same holds for $t \in \real_{\ge 0}$, then the attack is also
  unidentifiable by a dynamic monitor.
\end{lemma}

Lemma \ref{unidentifiable_input} follows analogously to Lemma
\ref{undetectable_input}. We now elaborate on the above lemmas to
derive fundamental detection and identification limitations for the
considered monitors.


\subsection{Fundamental limitations of static monitors}



Following Lemma \ref{undetectable_input}, an attack is undetectable by
a static monitor if and only if, for all $t \in \mathbb{N}_0$, there
exists a vector $\xi(t)$ such that $y(t) = C \xi(t)$. Notice that this
condition is compatible with \cite{YL-MKR-PN:09}, where an attack is
detected if and only if the residual $r(t) = y(t) - \tilde C \hat
x(t)$ is nonzero for some $t \in \mathbb N_0$, where $\hat x(t) =
C^{\dag} y(t)$.
In the following, let $\|v\|_0$ denote the number of nonzero
components of the vector $v$.

\begin{theorem}{\bf \emph{(Static detectability of cyber-physical attacks)}}
\label{Theorem: Static detectability of cyber-physical attacks}
For the cyber-physical descriptor system \eqref{eq:
  cyber_physical_fault} and an attack set $K$, the following
statements are equivalent: \begin{enumerate}
  \item the attack set $K$ is undetectable by a static monitor;
  \item there exists an attack mode $u_K (t)$ satisfying, for some $x
    (t)$ and at every $t \in \mathbb N_0$,
    \begin{align}\label{eq:cond2_static}
      C x(t) + D_K u_K(t) = 0.
    \end{align}
\end{enumerate}
Moreover, there exists an attack set $K$, with $|K| = k \in
\mathbb{N}_0$, undetectable by a static monitor if and only if there
exist $x \in \real^n$ such that $\| C x \|_0 =k$.
\end{theorem}
\smallskip

Before presenting a proof of the above theorem, we highlight that a
necessary and sufficient condition for the equation
\eqref{eq:cond2_static} to be satisfied is that $D_K u_{K}(t) =
u_{y,K}(t) \in \Image (C)$ at all times $t \in \mathbb N_0$, where
$u_{y,K}(t)$ is the vector of the last $p$ components of
$u_K(t)$. Hence, statement (ii) in Theorem \ref{Theorem: Static
  detectability of cyber-physical attacks} implies that {\em no} state
attack can be detected by a static detection procedure, and that an
undetectable output attack exists if and only if $\Image (D_{K}) \cap
\Image (C) \neq \{ 0 \}$.

\begin{pfof}{Theorem \ref{Theorem: Static detectability of
      cyber-physical attacks}}
  As previously discussed, the attack $K$ is undetectable by a static
  monitor if and only if for each $t \in \mathbb N$ there exists
  $x(t)$, and $u_{K}(t)$ such that
  \begin{align*}
    r(t) = y(t) - CC^{\dagger}
    y(t) = (I- C C^{\dagger}) \left( C x(t)
    + D_K u_K(t) \right)
  \end{align*}
  vanishes. Consequently, $r(t) = (I - C C^{\dagger}) D_K u_K(t)$, and
  the attack set $K$ is undetectable if and only if $D_K u_{K}(t) \in
  \Image(C)$, which is equivalent to statement (ii). The last
  necessary and sufficient condition in the theorem follows from (ii),
  and the fact that every output variable can be attacked
  independently of each other since $D = \begin{bmatrix}0 ,
    I\end{bmatrix}$.
\end{pfof}

We now focus on the static identification problem. Following Lemma
\ref{unidentifiable_input}, the following result can be asserted.

\begin{theorem}{\bf \emph{(Static identification of cyber-physical attacks)}}
  \label{Theorem: Static identifiability of cyber-physical attacks}
  For the cyber-physical descriptor system \eqref{eq:
    cyber_physical_fault} and an attack set $K$, the following
  statements are equivalent:
  \begin{enumerate}
  \item the attack set $K$ is unidentifiable by a static monitor;
  \item there exists an attack set $R$, with $|R|\le|K|$ and
    $R\neq{K}$, and attack modes $u_K (t)$, $u_R (t)$ satisfying, for
    some $x (t)$ and at every $t \in \mathbb N_0$,
    \begin{align*}
      C x(t) + D_K \left( u_K(t) + u_R(t) \right)= 0.
    \end{align*}
  \end{enumerate}
  Moreover, there exists an attack set $K$, with $|K| = k \in
  \mathbb{N}_0$, unidentifiable by a static monitor if and only if
  there exists an attack set $\bar K$, with $|\bar K| \le 2k$, which
  is undetectable by a static monitor.
\end{theorem}
\smallskip

Similar to the fundamental limitations of static detectability in
Theorem \ref{Theorem: Static detectability of cyber-physical attacks},
Theorem \ref{Theorem: Static identifiability of cyber-physical
  attacks} implies that, for instance, state attacks cannot be
identified and that an undetectable output attack of cardinality $k$
exists if and only if $\Image (D_{\bar K}) \cap \Image (C) \neq \{ 0
\}$, for some attack set $\bar K$ with $|\bar K| \le 2k$.
 
\begin{pfof}{Theorem \ref{Theorem: Static identifiability of cyber-physical attacks}}
  Due to linearity of the system \eqref{eq:
    cyber_physical_fault}, the unidentifiability condition in
  Lemma \ref{unidentifiable_input} is equivalent to
  $y(x_K-x_R,u_{K}-u_{R},t) = 0$, for some initial conditions $x_K$,
  $x_R$, and attack modes $u_K (t)$, $u_R(t)$. The equivalence between
  statements (i) and (ii) follows. The last statement follows from
  Theorem \ref{Theorem: Static detectability of cyber-physical
    attacks}.
\end{pfof}

\subsection{Fundamental limitations of dynamic
  monitors}\label{sec:limitations_dynamic}
As opposed to a static monitor, a dynamic monitor checks for the
presence of attacks at every time $t \in \real_{\geq 0}$.
Intuitively, a dynamic monitor is harder to mislead than a static
monitor. The following theorem formalizes this expected result.


\begin{theorem}{\bf \emph{(Dynamic detectability of cyber-physical
      attacks)}}\label{Theorem: Dynamic detectability of cyber-physical attacks}
  For the cyber-physical descriptor system \eqref{eq:
    cyber_physical_fault} and an attack set $K$, the following
  statements are equivalent:
  \begin{enumerate}
  \item the attack set $K$ is undetectable by a dynamic monitor;
  \item there exists an attack mode $u_K (t)$ satisfying, for some $x
    (0)$ and for every $t \in \real_{\ge 0}$,
    \begin{align*}
      E \dot x (t) &= Ax(t) + B_K u_K(t) \,,\\
      0 &= C x(t) + D_K u_K (t) \,;
    \end{align*}
\item there exist $s \in \complex$, $g \in \real^{|K|}$, and $x \in
  \real^{n}$, with $x \neq 0$, such that $(s E - A)x - B_K
  g = 0$ and $C x + D_K g = 0$.
  \end{enumerate}

  \noindent
  Moreover, there exists an attack set $K$, with $|K| = k$,
  undetectable by a dynamic monitor if and only if there exist $s \in
  \complex$ and $x \in \real^{n}$ such that $\|(sE - A) x \|_0 + \|C
  x\|_0 = k$.
\end{theorem}
\smallskip

Before proving Theorem \ref{Theorem: Dynamic detectability of
  cyber-physical attacks}, some comments are in order. First,
differently from the static case, state attacks {\em can be detected}
in the dynamic case. Second, in order to mislead a dynamic monitor an
attacker needs to inject a signal which is consistent with the system
dynamics at every instant of time. Hence, as opposed to the static
case, the condition $D_K u_{K}(t) = u_{y,K}(t) \in \Image (C)$ needs
to be satisfied for every $t \in \real_{\ge 0}$, and it is only
necessary for the undetectability of an output attack. Indeed, for
instance, state attacks can be detected even though they automatically
satisfy the condition $D_K u_{K}(t) = 0 \in \Image (C)$. Third and
finally, according to the last statement of Theorem \ref{Theorem:
  Dynamic detectability of cyber-physical attacks}, the existence of
invariant zeros\footnote{For the system $(E,A,B_K,C,D_K)$, the value
  $s \in \complex$ is an invariant zero if there exists $x \in
  \real^{n}$, with $x \neq 0$, $g \in \real^{|K|}$, such that $(sE-
  A)x - B_K g = 0$ and $C x + D_K g = 0$.}
for the system $(E,A,B_K, C, D_K)$ is equivalent to the existence of
undetectable attacks. As a consequence, a dynamic monitor performs
better than a static monitor, while requiring, possibly, fewer
measurements. We refer to Section \ref{sec:example_2} for an
illustrative example of this last statement.

\begin{pfof}{Theorem \ref{Theorem: Dynamic detectability of
      cyber-physical attacks}}
  By Lemma \ref{undetectable_input} and linearity of the system
  \eqref{eq: Kron-reduced model}, the attack mode
  $u_K(t)$ is undetectable by a dynamic monitor if and only if there
  exists $x_0$ such that $y(x_0,u_K,t) = 0$ for all $t \in \mathbb
  R_{\geq 0}$, that is, if and only if the system \eqref{eq:
    cyber_physical_fault} features zero dynamics. Hence, statements
  (i) and (ii) are equivalent.  For a linear descriptor system with
  smooth input and consistent initial condition, the existence of zero
  dynamics is equivalent to the existence of invariant zeros
  \cite[Theorem 3.2 and Proposition 3.4]{TG:93}. The equivalence of
  statements (ii) and (iii) follows.
  The last statement follows from (iii), and the fact that $B
  = \begin{bmatrix}I , 0\end{bmatrix}$ and $D = \begin{bmatrix}0 ,
    I\end{bmatrix}$.
\end{pfof}

We now consider the identification problem.

\begin{theorem}
  {\bf \emph{(Dynamic identifiability of cyber-physical attacks)}}
  \label{Theorem: Dynamic identifiability of cyber-physical attacks}
  For the cyber-physical descriptor system \eqref{eq:
    cyber_physical_fault} and an attack set $K$, the following
  statements are equivalent:
  \begin{enumerate}
  \item the attack set $K$ is unidentifiable by a dynamic monitor;
  \item there exists an attack set $R$, with $|R| \le |K|$ and
    $R\neq{K}$, and attack modes $u_K(t)$, $u_R(t)$ satisfying, for
    some $x (0)$ and for every $t \in \real_{\ge 0}$,
    \begin{align*}
      E \dot x (t) &= Ax(t) + B_K u_K(t) + B_R u_R(t) \,,\\
      0 &= C x(t) + D_K u_K (t) + D_R u_R (t) \,;
    \end{align*}
\item there exists an attack set $R$, with $|R| \le |K|$ and
  $R\neq{K}$, $s \in \mathbb{C}$, $g_K \in \mathbb{R}^{|K|}$, $g_R \in
  \mathbb{R}^{|R|}$, and $x \in \mathbb{R}^{n}$, with $x \neq 0$,
  such that $(sE - A) x - B_K g_K - B_R g_R = 0$
  and $C x + D_K g_K + D_R g_R = 0$.
  \end{enumerate}
  Moreover, there exists an attack set $K$, with $|K| = k \in
  \mathbb{N}_0$, unidentifiable by a dynamic monitor if and only if
  there exists an attack set $\bar K$, with $|\bar K| \le 2 k$, which
  is undetectable by a dynamic monitor.
\end{theorem}

\begin{pf}
  Notice that, because of the linearity of the system \eqref{eq:
    cyber_physical_fault}, the unidentifiability condition in
  Lemma \ref{unidentifiable_input} is equivalent to the condition
  $y(x_K-x_R,u_{K}-u_{R},t) = 0$, for some initial conditions $x_K$,
  $x_R$, and attack modes $u_K (t)$, $u_R(t)$. The equivalence between
  statements (i) and (ii) follows. Finally, the last two statements
  follow from Theorem \ref{Theorem: Dynamic detectability of
    cyber-physical attacks}, and the fact that $B = \begin{bmatrix}I ,
    0\end{bmatrix}$ and $D = \begin{bmatrix}0 , I\end{bmatrix}$.
\end{pf}

In other words, the existence of an unidentifiable attack set $K$ of
cardinality $k$ is equivalent to the existence of invariant zeros for
the system $(E,A,B_{\bar K},C,D_{\bar K})$, for some attack set $\bar
K$ with $|\bar K| \le 2k$. We conclude this section with the following
remarks. The existence condition in Theorem 3.4 is hard to verify
because of its combinatorial complexity: in order to check if there
exists an unidentifiable attack set $K$, with $|K| = k$, one needs to
certify the absence of invariant zeros for all possible
$2k$-dimensional attack sets. Thus, a conservative verification scheme
requires $\binom{n + p }{2k}$ tests. In Section
\ref{sec:graph_conditions} we present intuitive graph-theoretic
conditions for the existence of undetectable and unidentifiable attack
sets for a given sparsity pattern of the system matrices and generic
system parameters. Finally, Theorem \ref{Theorem: Dynamic
  identifiability of cyber-physical attacks} includes as a special
case Proposition 4 in \cite{FH-PT-SD:11}, which considers exclusively
output attacks.

\subsection{Fundamental limitations of active
  monitors}\label{sec:active}
An active monitor uses a control signal (unknown to the attacker) to
reveal the presence of attacks; see \cite{YM-BS:10a} for the case of
replay attacks.  In the presence of an active monitor with input
signal $w(t) = [w_x^\transpose(t) \; w_y^\transpose(t)]^\transpose$,
the system \eqref{eq: cyber_physical_fault} reads as
\begin{align*}
  \begin{split}
    E \dot x(t) &= Ax(t) + B_K u_K(t) + w_x(t),\\
    y(t) &= C x(t) + D_Ku_K(t) + w_y(t).
  \end{split}
\end{align*}
Although the attacker is unaware of the signal $w(t)$, active and
dynamic monitors share the same limitations.

\begin{theorem}{\bf \emph{(Limitations of active
      monitors)}}\label{thm:active_det}
  For the cyber-physical descriptor system \eqref{eq:
    cyber_physical_fault}, let $w(t)$ be an additive signal injected
  by an active monitor. The existence of undetectable (respectively
  unidentifiable) attacks does not depend upon the signal
  $w(t)$. Moreover, undetectable (respectively unidentifiable) attacks
  can be designed independently of $w(t)$.
\end{theorem}
\begin{proof}
  For the system \eqref{eq: cyber_physical_fault}, let $u(t)$ be the
  attack mode, and let $w(t)$ be the monitoring input. Let
  $y(x,u,w,t)$ denotes the output generated by the inputs $u(t)$ and
  $w(t)$ with initial condition $x = x_1 + x_2$. Observe that, because
  of the linearity of \eqref{eq: cyber_physical_fault}, we have
  $y(x,u,w,t) = y(x_1,u,0,t) + y(x_2,0,w,t)$, with consistent initial
  conditions $x_1$ and $x_2$. Then, an attack $u(t)$ is undetectable
  if and only if $y(x,u,w,t) = y(\bar x,0,w,t)$, or equivalently
  $y(x_1,u,0,t) + y(x_2,0,w,t) = y(\bar x_1,0,0,t) + y(x_2,0,w,t)$,
  for some initial conditions $x$ and $\bar x = \bar x_1 + x_2$. The
  statement follows, since, from the equality above, the detectability
  of $u(t)$ does not depend upon $w(t)$.
\end{proof}

As a consequence of Theorem \ref{thm:active_det}, the existence of
undetectable attacks is independent of the presence of known control
signals. Therefore, in a worst-case scenario, active monitors are as
powerful as dynamic monitors.
%
%
%
Since replay attacks are detectable by an active monitor
\cite{YM-BS:10a}, Theorem \ref{thm:active_det} shows that replay
attacks are not worst-case attacks.

\begin{remark}{\bf \emph{(Undetectable attacks in the presence of state
    and measurements noise)}}
The input $w(t)$ in Theorem \ref{thm:active_det} may represent sensors
and actuators noise. In this case, Theorem \ref{thm:active_det} states
that the existence of undetectable attacks for a noise-free system
implies the existence of undetectable attacks for the same system
driven by noise. The converse does not hold, since attackers may
remain undetected by injecting a signal compatible with the noise
statistics.
\oprocend
\end{remark}

\subsection{Specific results for index-one singular systems}
For many interesting real-world descriptor systems, including the
examples in Section \ref{example:power} and \ref{example:water}, the
algebraic system equations can be solved explicitly, and the
descriptor system \eqref{eq: cyber_physical_fault} can be reduced to a
nonsingular state space system. For this reason, this section presents
specific results for the case of \emph{index-one} systems
\cite{FLL:86}. In this case, without loss of generality, we assume the
system \eqref{eq: cyber_physical_fault} to be written in the canonical
form
\begin{align}
  \label{eq:weierstrass}
  \begin{split}
  \begin{bmatrix}
    E_{11} & 0\\
    0 & 0
  \end{bmatrix}
  \begin{bmatrix}
    \dot x_1 \\ \dot x_2
  \end{bmatrix}
  &=
  \begin{bmatrix}
    A_{11} & A_{12}\\
    A_{21} & A_{22}
  \end{bmatrix}
  \begin{bmatrix}
    x_1 \\ x_2
  \end{bmatrix}
  +
  \begin{bmatrix}
    B_1 \\ B_2
  \end{bmatrix}
  u_K(t),\\
  y(t) &=
  \begin{bmatrix}
    C_1 & C_2
  \end{bmatrix}
  \begin{bmatrix}
    x_1 \\ x_2
  \end{bmatrix}
  +
  D_K u_K(t),
  \end{split}
\end{align}
where $E_{11}$ is nonsingular and $A_{22}$ is nonsingular. 
Consequently, the state $x_1$ and $x_2$
are referred to as \emph{dynamic state} and \emph{algebraic state},
respectively. The algebraic state can be
expressed via the dynamic state and the attack mode as
\begin{align}\label{eq: elimination of bus voltages}
  x_2(t) = -A_{22}^{-1} A_{21} x_1 (t) - A_{22}^{-1} B_2 u_K(t).
\end{align}
The elimination of the algebraic state $x_2$ in the descriptor system
\eqref{eq:weierstrass} leads to the nonsingular state space system
\begin{align}
  \dot x_1 =&\; \underbrace{E_{11}^{-1}\left( A_{11} - A_{12}A_{22}^{-1}A_{21}
    \right)}_{\tilde A} x_1(t) \nonumber\\
    &\;+ \underbrace{ E_{11}^{-1}\left( B_1 - A_{12}
      A_{22}^{-1} B_2 \right)}_{\tilde B_K} u_K(t), \label{eq: Kron-reduced model}\\
  y(t) =&\; \underbrace{\left( C_1 - C_2 A_{22}^{-1}
      A_{21}\right)}_{\tilde C} x_1(t) + \underbrace{\left( D_K - C_2
      A_{22}^{-1} B_2 \right)}_{\tilde D_K} u_K(t). \nonumber
\end{align}

This reduction of the algebraic states is known as Kron reduction in
the literature on power networks and circuit theory
\cite{FD-FB:11d}. Hence, we refer to \eqref{eq: 
  Kron-reduced model} as the {\em Kron-reduced system}. 
  
Clearly, for any state trajectory $x_{1}(t)$ of the Kron-reduced
system \eqref{eq: Kron-reduced model}, the corresponding state
trajectory $[x_{1}^\transpose(t) \;x_{2}^\transpose(t)]^\transpose$ of
the (non-reduced) cyber-physical descriptor system \eqref{eq:
  cyber_physical_fault} can be recovered by identity \eqref{eq:
  elimination of bus voltages} and given knowledge of the input
$u_K(t)$.
The following subtle issues are easily visible in the Kron-reduced
system \eqref{eq: elimination of bus voltages}. First, a state attack
affects directly the output $y(t)$, provided that $C_2 A_{22}^{-1}B_2
u_K(t) \neq 0$. Second, since the matrix $A_{22}^{-1}$ is generally
fully populated, an attack on a single algebraic component can affect
not only the locally attacked state or its vicinity but larger parts
of the system.

According to the transformations in \eqref{eq: Kron-reduced
  model}, for each attack set $K$, the attack signature $(B_K,D_K)$ is mapped to
the corresponding signature $(\tilde B_K, \tilde D_K)$ in the
Kron-reduced system. 
%
%
As an apparent disadvantage, the sparsity pattern of the original
(non-reduced) cyber-physical descriptor system \eqref{eq:
  cyber_physical_fault} is lost in the Kron-reduced representation
\eqref{eq: Kron-reduced model}, and so is, possibly, the physical
interpretation of the state and the direct representation of system
components. However, as we show in the following lemma, the notions of
detectability and identifiability of an attack set $K$ defined for the
original descriptor system \eqref{eq: cyber_physical_fault} are
equivalent for the Kron-reduced system \eqref{eq: Kron-reduced
  model}. This property renders the low-dimensional and nonsingular
Kron-reduced system \eqref{eq: Kron-reduced model} attractive from a
computational point of view to design attack detection and
identification monitors; see \cite{FP-FD-FB:10yb}.

\begin{lemma}{\bf \emph{(Equivalence of detectability and
      identifiability under Kron reduction)}}
\label{lemma:equivalence}
For the cyber-physical descriptor system \eqref{eq:
  cyber_physical_fault}, the attack set $K$ is detectable
(respectively identifiable) if and only if it is detectable
(respectively identifiable) for the associated Kron-reduced system \eqref{eq:
  Kron-reduced model}.
\end{lemma}
\begin{proof}
  The lemma follows from the fact that the input and initial condition
  to output map for the system \eqref{eq: cyber_physical_fault}
  coincides with the corresponding map for the Kron-reduced system
  \eqref{eq: Kron-reduced model} and equation \eqref{eq: elimination
    of bus voltages}. Indeed, according to Theorem \ref{Theorem:
    Dynamic detectability of cyber-physical attacks}, the attack set
  $K$ is undetectable if and only if there exist $s \in \complex$, $g
  \in \real^{|K|}$, and $x = [x_{1}^{\transpose} \;
  x_{2}^{\transpose}]^{\transpose} \in \real^{n}$, with $x \neq 0$,
  such that
\begin{equation*}
	(s E - A)x - B_K g = 0 \mbox{ and } C x + D_K g = 0 \,.
\end{equation*}
Equivalently, by eliminating the algebraic constraints as in
\eqref{eq: elimination of bus voltages}, the attack set $K$ is
undetectable if and only if the conditons
\begin{equation*}
	(s I - \tilde A)x_{1} - \tilde B_K g = 0 \mbox{ and } \tilde C x_{1} + \tilde D_K g = 0 
\end{equation*}
are satisfied together with $ x_2 = -A_{22}^{-1} A_{21} x_1 -
A_{22}^{-1} B_2 g$. Notice that the latter equation is always
satisfied due to the consistency assumption (A2), and the equivalence
of detectability of the attack set $K$ follows. The equivalence of
attack identifiability follows by analogous arguments.
\end{proof}

\subsection{Attack detection and identification in presence of
  inconsistent initial conditions and impulsive attack
  signals}\label{Subsection: smoothness issues}
We now discuss the case of non-smooth attack signal and inconsistent
initial condition.
If the consistency assumption (A3) is dropped, then discontinuities in
the state $x(t \downarrow 0)$ may affect the measurements $y(t
\downarrow 0)$. For instance for index-one systems, an inconsistent
initial condition leads to an initial jump for the algebraic variable
$x_{2}(t \downarrow 0)$ to obey equation \eqref{eq: elimination of bus
  voltages}. Consequently, the inconsistent initial value
$[0^{\transpose} \; x_{2}(0)^{\transpose}]^{\transpose} \in \Ker(E)$
cannot be recovered through measurements.

Assumption (A4) requires the attack signal to be sufficiently smooth
such that $x(t)$ and $y(t)$ are at least continuous.  Suppose that
assumption (A4) is dropped and the input $u(t)$ belongs to the class
of impulsive smooth distributions $\mathcal C_{\textup{imp}} =
\mathcal C_{\textup{smooth}} \cup \mathcal C_{\textup{p-imp}}$, that
is, loosely speaking, the class of functions given by the linear
combination of a smooth function on $\mathbb R_{\geq 0}$ (denoted by
$\mathcal C_{\textup{smooth}}$) and Dirac impulses and their
derivatives at $t=0$ (denoted by $\mathcal C_{\textup{p-imp}}$), see
\cite{TG:93},\cite[Section 2.4]{PK-VLM:06}. In this case, an attacker
commanding an impulsive input $u(0)\in \mathcal C_{\textup{imp}}$ can
reset the initial state $x(0)$ and, possibly, evade detection.

The discussion in the previous two paragraphs can be formalized as
follows.
Let $\mathcal V_{c}$ be the subspace of points $x_{0} \in \mathbb
R^{n}$ of consistent initial conditions for which there exists an
input $u \in \mathcal C^{m}_{\textup{smooth}}$ and a state trajectory
$x \in \mathcal C^{n}_{\textup{smooth}}$ to the descriptor system
\eqref{eq: cyber_physical_fault} such that $y(t) = 0$ for all $t \in
\real_{\geq 0}$.
Let $\mathcal V_{d}$ (respectively $\mathcal W$) be the subspace of
points $x_{0} \in \mathbb R^{n}$ for which there exists an input $u
\in \mathcal C^{n+p}_{\textup{imp}}$ (respectively $u \in \mathcal
C^{n+p}_{\textup{p-imp}}$) and a state trajectory $x \in \mathcal
C^{n}_{\textup{imp}}$ (respectively $x \in \mathcal
C^{n}_{\textup{p-imp}}$) to the descriptor system \eqref{eq:
  cyber_physical_fault} such that $y(t) = 0$ for all $t \in
\real_{\geq 0}$.
The output-nulling subspace $\mc V_{d}$ can be decomposed as follows:
   
\begin{lemma}
  \label{Lemma: Decomposition of output-nulling subspace}
  {\bf \emph{(Decomposition of output-nulling space \cite[Theorem 3.2
      and Proposition 3.4]{TG:93}))}} $\mathcal V_{d} = \mathcal V_{c}
  + \mathcal W + \Ker(E)$.
\end{lemma}
In words, from an initial condition $x(0) \in \mathcal V_{d}$ the
output can be nullified by a smooth input or by an impulsive input
(with consistent or inconsistent initial conditions in
$\Ker(E)$).
   
In this work we focus on the smooth output-nulling subspace $\mathcal
V_{c}$, which is exactly space of zero dynamics identified in Theorems
\ref{Theorem: Dynamic detectability of cyber-physical attacks} and
\ref{Theorem: Dynamic identifiability of cyber-physical attacks}.
Hence, by Lemma \ref{Lemma: Decomposition of output-nulling subspace},
for inconsistent initial conditions, the results presented in this
section are valid only for strictly positive times $t>0$. On the
other hand, if an attacker is capable of injecting impulsive signals,
then it can avoid detection for initial conditions $x(0) \in \mathcal
W$.

\section{Graph theoretic detectability
  conditions}\label{sec:graph_conditions}
In this section we characterize undetectable attacks against
cyber-physical systems from a structural perspective. In particular we
will derive detectability conditions based upon a connectivity
property of a graph associated with the system. For ease of notation,
we now drop the subscript $K$ from $B_K$, $D_K$, and $u_K(t)$.

\subsection{Preliminary notions}
We start by recalling some useful facts about structured systems and
structural properties \cite{KJR:88,WMW:85}. Let a \emph{structure
  matrix} $[M]$ be a matrix in which each entry is either a fixed zero
or an indeterminate parameter. The system
\begin{align}
  \label{eq:structural_descriptor}
  \begin{split}
    [E] \dot x(t) &= [A] x(t) + [B] u(t),\\
    y(t) &= [C] x(t) + [D] u(t).
  \end{split}
\end{align}
is called \emph{structured system}, and it is sometimes referred to
with the tuple $([E],[A],[B],[C],[D])$ of structure matrices. A system
$(E,A,B,C,D)$ is an admissible realization of $([E],[A],[B],[C],[D])$
if it can be obtained from the latter by fixing the indeterminate
entries at some particular value. Two systems are structurally
equivalent if they are both an admissible realization of the same
structured system. Let $d$ be the number of indeterminate entries of a
structured system altogether. By collecting the indeterminate
parameters into a vector, an admissible realization is mapped to a
point in the Euclidean space $\real^d$. A property which can be
asserted on a dynamical system is called \emph{structural} if,
informally, it holds for \emph{almost all} admissible realizations. To
be more precise, we say that a property is structural if and only if
the set of admissible realizations satisfying such property forms a
dense subset of the parameters space.\footnote{A subset $S \subseteq P
  \subseteq \real^d$ is dense in $P$ if, for each $r \in P$ and every
  $\varepsilon > 0$, there exists $s \in S$ such that the Euclidean
  distance $\|s - r\| \le \varepsilon$.} For instance,
left-invertibility of a nonsingular system is a structural property
with respect to $\real^d$ \cite{JMD-CC-JW:03}.

Consider the structured cyber-physical system
\eqref{eq:structural_descriptor}. It is often the case that, for the
tuple $(E,A,B,C,D)$ to be an admissible realization of
\eqref{eq:structural_descriptor}, the numerical entries need to
satisfy certain algebraic relations. For instance, for $(E,A,B,C,D)$
to be an admissible power network realization, the matrices $E$ and
$A$ need to be of the form \eqref{eq: power network descriptor system
  model}. Let $\mathbb S \subseteq \mathbb{R}^d$ be the admissible
parameter space. We make the following assumption:
\begin{itemize}
\item[(A4)] the admissible parameters space $\mathbb S$ is a polytope
  of $\mathbb R^d$, that is, $\mathbb S = \setdef{x \in \real^d}{M x
    \ge 0}$ for some matrix $M$.
\end{itemize}
It should be noticed that assumption (A4) is automatically verified
for the case of power networks \cite[Lemma
3.1]{FP-AB-FB:10u}. Unfortunately, if the admissible parameters space
is a subset of $\real^d$, then classical structural system-theoretic
results are, in general, not valid \cite[Section 15]{KJR:88}.

We now define a mapping between dynamical systems in descriptor form
and digraphs. Let ($[E]$,$[A]$,$[B]$,$[C]$,$[D]$) be a structured
cyber-physical system under attack. We associate a directed graph
$G=(\V,\mathcal{E})$ with the tuple
($[E]$,$[A]$,$[B]$,$[C]$,$[D]$). The vertex set is $\V = \mathcal{U}
\cup \mathcal{X} \cup \mathcal{Y}$, where
$\mathcal{U}=\{u_1,\dots,u_{m}\}$ is the set of input vertices,
$\mathcal{X}=\{x_1,\dots,x_{n}\}$ is the set of state vertices, and
$\mathcal{Y}=\{y_1,\dots,y_{p}\}$ is the set of output vertices. If
$(i,j)$ denotes the edge from the vertex $i$ to the vertex $j$, then
the edge set $\mathcal{E}$ is $\mathcal{E}_{[E]} \cup
\mathcal{E}_{[A]} \cup \mathcal{E}_{[B]} \cup \mathcal{E}_{[C]} \cup
\mathcal{E}_{[D]}$, with $\mathcal{E}_{[E]}=\{(x_j,x_i) : [E]_{ij}\neq
0\}$, $\mathcal{E}_{[A]}=\{(x_j,x_i) : [A]_{ij}\neq 0\}$,
$\mathcal{E}_{[B]}=\{(u_j,x_i) : [B]_{ij}\neq 0\}$,
$\mathcal{E}_{[C]}=\{(x_j,y_i) : [C]_{ij}\neq 0\}$, and
$\mathcal{E}_{[D]}=\{(u_j,y_i) : [D]_{ij}\neq 0\}$. In the latter, for
instance, the expression $[E]_{ij} \neq 0$ means that the $(i,j)$-th
entry of $[E]$ is a free parameter.

\begin{example}{\bf \emph{(Power network structural
      analysis)}}\label{Example: power network structural analysis}
  \begin{figure}[tb!]
    \centering
    \includegraphics[width=.7\columnwidth]{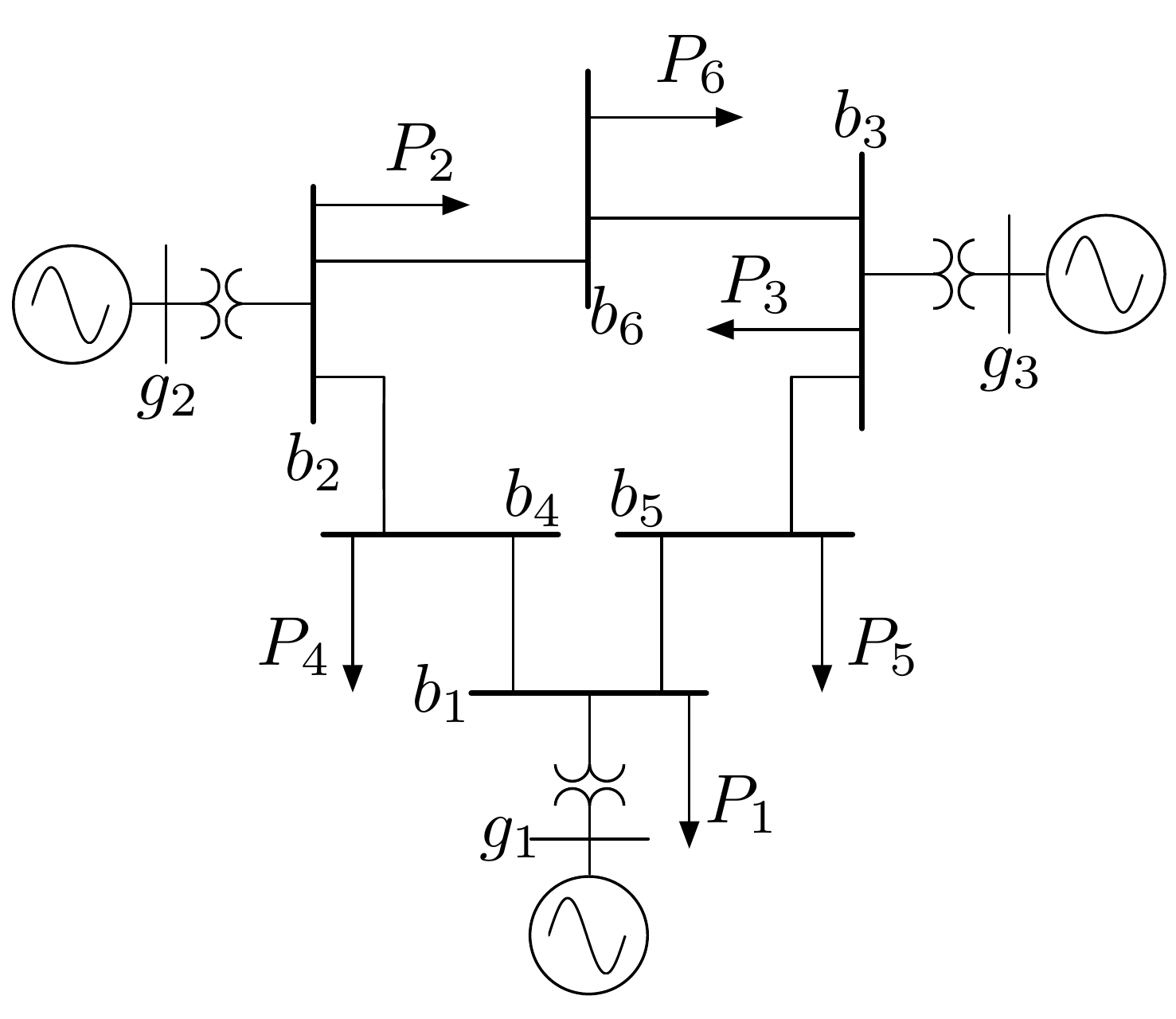}\\
    \caption{WSSC power system with $3$ generators and $6$ buses. The
      numerical value of the network parameters can be found in
      \cite{ES:04}.}\label{power_network}
  \end{figure}
  \begin{figure}
    \centering
    \includegraphics[width=.7\columnwidth]{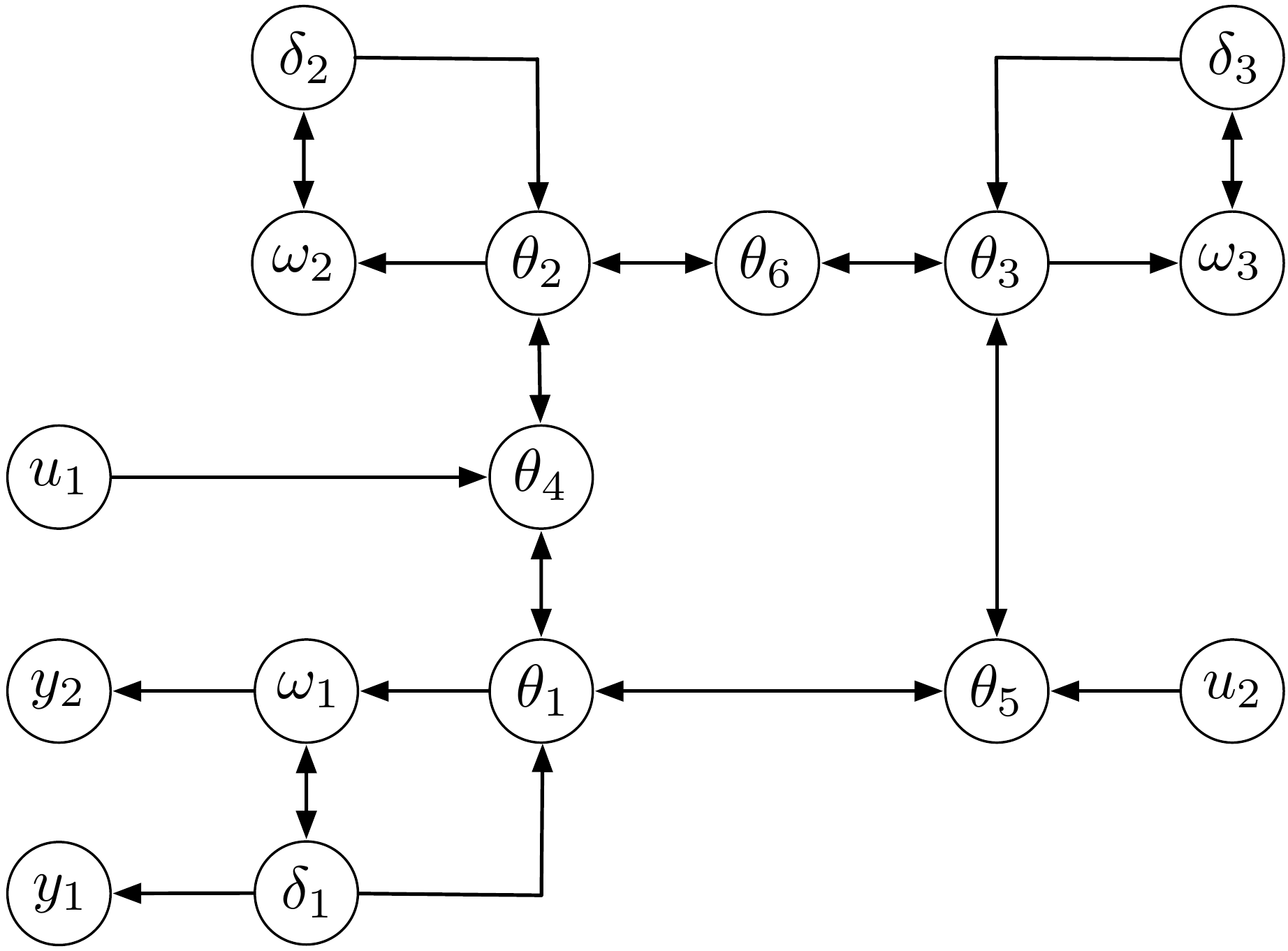}\\
    \caption{The digraph associated with the network in
      Fig. \ref{power_network}. The self-loops of the vertices
      $\{\delta_1,\delta_2,\delta_3\}$,
      $\{\omega_1,\omega_2,\omega_3\}$, and
      $\{\theta_1,\dots,\theta_6\}$ are not drawn. The inputs $u_1$
      and $u_2$ affect respectively the bus $b_4$ and the bus
      $b_5$. The measured variables are the rotor angle and frequency
      of the first generator.}\label{digraph}
  \end{figure}
  Consider the power network illustrated in Fig. \ref{power_network},
  where, being $e_i$ the $i$-th canonical vector, we take
  $[E]=\text{blkdiag}(1,1,1,M_1,M_2,M_3,0,0,0,0,0,0)$, $[B]=[e_8 \;
  e_9]$, $[C]=[e_1 \; e_4]^\transpose$, $[D] = 0$, and $[A]$ equal to
\begin{align*}
      \left[
    \begin{smallmatrix}
      0 & 0 & 0 & 1 & 0 & 0 & 0 & 0 & 0 & 0 & 0 & 0\\
      0 & 0 & 0 & 0 & 1 & 0 & 0 & 0 & 0 & 0 & 0 & 0\\
      0 & 0 & 0 & 0 & 0 & 1 & 0 & 0 & 0 & 0 & 0 & 0\\
      a_{4,1} & 0 & 0 & a_{4,4} & 0 & 0 & a_{4,7} & 0 & 0 & 0 & 0 & 0\\
      0 & a_{5,2} & 0 & 0 & a_{5,5} & 0 & 0 & a_{5,8} & 0 & 0 & 0 & 0\\
      0 & 0 & a_{6,3} & 0 & 0 & a_{6,6} & 0 & 0 & a_{6,9} & 0 & 0 & 0\\
      a_{7,1} & 0 & 0 & 0 & 0 & 0 & a_{7,7} & 0 & 0 & a_{7,10} & a_{7,11} & 0\\
      0 & a_{8,2} & 0 & 0 & 0 & 0 & 0 & a_{8,8} & 0 & a_{8,10} & 0 & a_{8,12}\\
      0 & 0 & a_{9,3} & 0 & 0 & 0 & 0 & 0 & a_{9,9} & 0 & a_{9,11} & a_{9,12}\\
      0 & 0 & 0 & 0 & 0 & 0 & a_{10,7} & a_{10,8} & 0 & a_{10,10} & 0 & 0\\
      0 & 0 & 0 & 0 & 0 & 0 & a_{11,7} & 0 & a_{11,9} & 0 & a_{11,11} & 0\\
      0 & 0 & 0 & 0 & 0 & 0 & 0 & a_{12,8} & a_{12,9} & 0 & 0 & a_{12,12}
      \end{smallmatrix}
    \right]
\end{align*}
The digraph associated with the structure matrices
$([E],[A],[B],[C],[D])$ is shown in Fig. \ref{digraph}.  \oprocend
\end{example}

\subsection{Network vulnerability with known initial
  state}\label{sec:left_inv}
We derive graph-theoretic detectability conditions for two different
scenarios. Recall from Lemma \ref{undetectable_input} that an
attack $u(t)$ is undetectable if $y(x_1,u,t) = y(x_2,0,t)$ for some
initial states $x_1$ and $x_2$. In this section, we assume that the
system state is known at the failure initial time,\footnote{The
  failure initial state can be estimated through a state observer
  \cite{ES:04}.} so that an attack $u(t)$ is undetectable if
$y(x_0,u,t) = y(x_0,0,t)$ for some system initial state $x_0$. The
complementary case of unknown initial state is studied in Section
\ref{sec:inv_zeros}.

Consider the cyber-physical system described by the matrices
$(E,A,B,C,D)$, and notice that, if the initial state is known, then
the attack undetectability condition $y(x_0,u,t) = y(x_0,0,t)$
coincides with the system being not left-invertible.\footnote{A
  regular descriptor system is left-invertible if and only if its
  transfer matrix $G(s)$ is of full column rank for all almost all $s
  \in \mathbb{C}$, or if and only if $\left[\begin{smallmatrix}
      s  E  -  A  &  -B \\
      C & D
  \end{smallmatrix}\right]$ has full column rank for almost all $s \in
\mathbb{C}$ \cite[Theorem 4.2]{TG:93}.} Recall that a subset $S
\subseteq \real^d$ is an \emph{algebraic variety} if it coincides with
the locus of common zeros of a finite number of polynomials
\cite{WMW:85}. 
Consider the following
observation.

\begin{lemma}{\bf \emph{(Polytopes and algebraic
      varieties)}}\label{lemma:algebraic}
  Let $S \subseteq \real^d$ be a polytope, and let $T \subseteq
  \real^d$ be an algebraic variety. Then, either $S \subseteq T$, or
  $S \setminus (S \cap T)$ is dense in $S$.
\end{lemma}
\begin{proof}
  Let $T \subseteq \real^d$ be the algebraic variety described by the
  locus of common zeros of the polynomials
  $\{\phi_1(x),\dots,\phi_t(x)\}$, with $t \in \mathbb N$, $t <
  \infty$. Let $P \subseteq \real^d$ be the smallest vector subspace
  containing the polytope $S$. Then $P \subseteq T$ if and only if
  every polynomial $\phi_i$ vanishes identically on $P$. Suppose that
  the polynomial $\phi_i$ does not vanish identically on $P$. Then,
  the set $T \cap P$ is contained in the algebraic variety $\{x \in P
  : \phi_i(x) = 0\}$, and, therefore \cite{WMW:85}, the complement $P
  \setminus (P \cap T)$ is dense in $P$. By definition of a dense set,
  the set $S \setminus (S \cap T)$ is also dense in $S$.
\end{proof}

In Lemma \ref{lemma:algebraic} interpret the polytope $S$ as the
admissible parameters space of a structured cyber-physical system.
Then we have shown that left-invertibility of a cyber-physical system
is a structural property even when the admissible parameters space is
a polytope of the whole parameters space. Consequently, given a
structured cyber-physical system, either every admissible realization
admits an undetectable attack, or there is no undetectable attack in
almost all admissible realizations. Moreover, in order to show that
almost all realizations have no undetectable attacks, it is sufficient
to prove that this is the case for some specific admissible
realizations. Before presenting our main result, we recall the
following result. Let $\bar E$ and $\bar A$ be $N$-dimensional square
matrices, and let $G(s\bar E-\bar A)$ be the graph associated with the
matrix $s \bar E - \bar A$ that consists of $N$ vertices, and an edge
from vertex $j$ to $i$ if $\bar A_{ij} \neq 0$ or $\bar E_{ij} \neq
0$. The matrix $s[\bar E] - [\bar A]$ is said to be structurally
degenerate if, for any admissible realization $\bar E$ (respectively
$\bar A$) of $[\bar E]$ (respectively $[\bar A]$), the determinant
$|s\bar E - \bar A|$ vanishes for all $s \in \mathbb{C}$. Recall the
following definitions from \cite{JMD-CC-JW:03}. For a given graph $G$,
a path is a sequence of vertices where each vertex is connected to the
following one in the sequence. A path is simple if every vertex on the
path (except possibly the first and the last vertex) occurs only
once. Two paths are disjoint if they consist of disjoint sets of
vertices. A set of $l$ mutually disjoint and simple paths between two
sets of vertices $S_1$ and $S_2$ is called a \emph{linking} of size
$l$ from $S_1$ to $S_2$. A simple path in which the first and the last
vertex coincide is called cycle; a \emph{cycle family} of size $l$ is
a set of $l$ mutually disjoint cycles. The length of a cycle family
equals the total number of edges in the family.

\begin{theorem}{\bf \emph{(Structural rank of a square matrix
      \cite{KJR:94})}}\label{structural_rank}
  The structure $N$-dimensional matrix $s[\bar E] - [\bar A]$ is
  structurally degenerate if and only if there exists no cycle family
  of length $N$ in $G(s[\bar E] - [\bar A])$.
\end{theorem}

We are now able to state our main result on structural detectability.

\begin{theorem}{\bf \emph{(Structurally undetectable
    attack)}}\label{thm:vulnerability} Let the parameters space of the
structured cyber-physical system $([E],[A],[B],[C],[D])$ define a
polytope in $\real^d$ for some $d \in \mathbb{N}_0$. Assume that $s[E]
- [A]$ is structurally non-degenerate. The system
$([E],[A],[B],[C],[D])$ is structurally left-invertible if and only if
there exists a linking of size $|\mathcal{U}|$ from $\mathcal{U}$ to
$\mathcal{Y}$.
\end{theorem}

Theorem \ref{thm:vulnerability} can be interpreted in the context of
cyber-physical systems. Indeed, since $|sE-A| \neq 0$ by assumption
(A1), and because of assumption (A4), Theorem \ref{thm:vulnerability}
states that there exists a structural undetectable attack if and only
if there is no linking of size $|\mathcal{U}|$ from $\mathcal{U}$ to
$\mathcal{Y}$, provided that the network state at the failure time is
known.

\begin{proof}
  Because of Lemma \ref{lemma:algebraic}, we need to show that, if
  there are $|\mathcal{U}|$ disjoint paths from $\mathcal{U}$ to
  $\mathcal{Y}$, then there exists admissible left-invertible
  realizations. Conversely, if there are at most $|\mathcal{U}|-1$
  disjoint paths from $\mathcal{U}$ to $\mathcal{Y}$, then every
  admissible realization is not left-invertible.

  \emph{(If)} Let $(E,A,B,C,D)$, with $|sE-A| \neq 0$, be
  an admissible realization, and suppose there exists a linking of
  size $|\mathcal{U}|$ from $\mathcal{U}$ to $\mathcal{Y}$. Without
  affecting generality, assume $|\mathcal{Y}| = |\mathcal{U}|$.
  For the left-invertibility property we need 
  \begin{align*}
    \left|
      \begin{bmatrix}
        s  E  -  A  &  -B \\
        C & D
      \end{bmatrix}
    \right| = \left|sE-A\right| \left|D + C (sE-A)^{-1} B\right|
  \neq 0,
  \end{align*}
  and hence we need $\left|D + C (sE-A)^{-1} B\right| \neq 0$. Notice
  that $D + C (sE-A)^{-1} B$ corresponds to the transfer matrix of the
  cyber-physical system. Since there are $|\mathcal{U}|$ independent
  paths from $\mathcal{U}$ to $\mathcal{Y}$, the matrix $D + C
  (sE-A)^{-1} B$ can be made nonsingular and diagonal by removing some
  connection lines from the network. In particular, for a given
  linking of size $|\mathcal{U}|$ from $\mathcal{U}$ to $\mathcal{Y}$,
  a nonsingular and diagonal transfer matrix is obtained by setting to
  zero the entries of $E$ and $A$ corresponding to the edges not in
  the linking. Then there exist admissible left-invertible
  realizations, and thus the system $([E],[A],[D],[C],[D])$ is
  structurally left-invertible.

  \emph{(Only if)} Take any subset of $|\mathcal{U}|$ output vertices,
  and let $|\mathcal{U}|-1$ be the maximum size of a linking from
  $\mathcal{U}$ to $\mathcal{Y}$. Let $[\bar E]$ and $[\bar A]$ be
  such that $s[\bar E] - [\bar A] = \left[
  \begin{smallmatrix}
    s[E] -[A] & [B]\\
    [C] & [D]
  \end{smallmatrix}
\right]$. Consider the previously defined graph $G(s [\bar E]- [ \bar
A])$, and notice that a path from $\mathcal{U}$ to $\mathcal{Y}$ in
the digraph associated with the structured system corresponds,
possibly after relabeling the output variables, to a cycle in
involving input/output vertices in $G(s [\bar E]- [ \bar A])$. Observe
that there are only $|\mathcal{U}|-1$ such (disjoint) cycles. Hence,
there is no cycle family of length $N$, being $N$ the size of $[\bar
A]$, and the statement follows from Theorem \ref{structural_rank}.
%
\end{proof}

To conclude this section, note that Theorem \ref{thm:vulnerability}
extends \cite{JWW:91} to regular descriptor systems with constraints
on parameters.

\subsection{Network vulnerability with unknown initial
  state}\label{sec:inv_zeros} 
If the failure initial state is unknown, then a vulnerability is
identified by the existence of a pair of initial conditions $x_1$ and
$x_2$, and an attack $u(t)$ such that $y(x_1,0,t) = y(x_2,u,t)$, or,
equivalently, by the existence of invariant zeros for the given
cyber-physical system. We will now show that, provided that a
cyber-physical system is left-invertible, its invariant zeros can be
computed by simply looking at an associated nonsingular state space
system. Let the state vector $x$ of the descriptor system \eqref{eq:
  cyber_physical_fault} be partitioned as $[x_1^\transpose \;
x_2^\transpose]^\transpose$, where $x_1$ corresponds to the dynamic
variables. Let the network matrices $E$, $A$, $B$, $C$, and $D$ be
partitioned accordingly, and assume, without loss of generality, that
$E$ is given as $E = \textup{blkdiag}(E_{11},0)$, where $E_{11}$ is
nonsingular. In this case, the descriptor model \eqref{eq:
  cyber_physical_fault} reads\,as
\begin{align}
\label{eq:partitioned_descriptor_system}
	\begin{split}
    E_{11} \dot x_1 (t)&= A_{11} x_1(t) + B_1 u(t) + A_{12} x_2(t)\,,\\
    0 &= A_{21} x_{1}(t) + A_{22} x_{2}(t) + B_{2} u(t)\,, \\
    y(t) &= C_{1}x_{1}(t) + C_{2} x_{2}(t)  + Du(t)
    \,.
  \end{split}
\end{align}
Consider now the associated nonsingular state space system which is
obtained by regarding $x_{2}(t)$ as an external input to the
descriptor system \eqref{eq:partitioned_descriptor_system} and the
algebraic constraint as output:
\begin{align}
	\label{eq:associated_nonsingular_system}
	\begin{split}
  \dot x_1 (t)&= E_{11}^{-1}A_{11} x_1(t) + E_{11}^{-1} B_1 u(t) +E_{11}^{-1} A_{12} x_2(t),\\
  \tilde y(t) & =
  \begin{bmatrix}
    A_{21}\\
    C_1
  \end{bmatrix}
  x_1(t) + 
  \begin{bmatrix}
    A_{22}  & B_2\\
    C_2 & D
  \end{bmatrix}
  \begin{bmatrix}
    x_2(t)\\
    u(t)
  \end{bmatrix}
    \,.
  \end{split}
\end{align}

\begin{theorem}{\bf \emph{(Equivalence of invariant
      zeros)}}\label{thm:inv_zero}
  Consider the descriptor system \eqref{eq: cyber_physical_fault}
  partitioned as in \eqref{eq:partitioned_descriptor_system}.
  Assume that, for the corresponding structured system
  $([E],[A],[B],[C],[D])$, there exists a linking of size
  $|\mathcal{U}|$ from $\mathcal{U}$ to $\mathcal{Y}$.
  Then, in almost all admissible realizations, the invariant zeros of
  the descriptor system \eqref{eq:partitioned_descriptor_system}
  coincide with those of the associated nonsingular system
  \eqref{eq:associated_nonsingular_system}.
\end{theorem}
\begin{proof}
  From Theorem \ref{thm:vulnerability}, the structured descriptor
  system $([E],[A],[B],[C],[D])$ is structurally left-invertible. Let
  $(E,A,B,C,D)$ be a left-invertible realization.
  
  The proof now follows a procedure similar to \cite[Proposition
  8.4]{JT:06}. Let $s \in \mathbb C$ be an invariant zero for the
  nonsingular system \eqref{eq:associated_nonsingular_system} with
  state-zero direction $x_1 \neq 0$ and input-zero direction $u$, that
  is
\begin{equation*}
\begin{bmatrix}
0 \\ 0 \\ 0
\end{bmatrix}
=
\underbrace{
\left[\begin{array}{c|cc}
sI - E_{11}^{-1}A_{11} & -E_{11}A_{12} & - E_{11}^{-1} B_{1} \\ \hline
A_{21} & A_{22} & B_{2}\\
C_{1} & C_{2} & D
\end{array}\right]
}_{\subscr{P}{nonsingular}(s)}
\begin{bmatrix}
x_{1} \\ x_{2} \\ u
\end{bmatrix}
\,.
\end{equation*}
A multiplication of the above equation by
$\textup{blkdiag}(E_{11},-I,I)$ and a re-partioning of the resulting
matrix yields
\begin{equation}
\begin{bmatrix}
0 \\ 0 \\ 0
\end{bmatrix}
=
\underbrace{
\left[\begin{array}{cc|c}
sE_{11} - A_{11} & -A_{12} & - B_{1} \\ 
-A_{21} & -A_{22} & -B_{2}\\\hline
C_{1} & C_{2} & D
\end{array}\right]
}_{\subscr{P}{singular}(s)}
\begin{bmatrix}
x_{1} \\ x_{2} \\ u
\end{bmatrix}
\label{eq:partitioned_descriptor_system_pencil}
\,.
\end{equation}
Since $x_1 \neq 0$, we also have $x= [x_{1}^{\transpose} \;
x_{2}^{\transpose}]^{\transpose} \neq 0$. Then, equation
\eqref{eq:partitioned_descriptor_system_pencil} implies that $s \in
\mathbb C$ is an invariant zero of the descriptor system
\eqref{eq:partitioned_descriptor_system} with state-zero direction $x
\neq 0$ and input-zero direction $u$. We conclude that the invariant
zeros of the nonsingular system
\eqref{eq:associated_nonsingular_system} are a subset of the zeros of
the descriptor system \eqref{eq:partitioned_descriptor_system}.
In order to continue, suppose that there is $s \in \mathbb C$ which is
an invariant zero of the descriptor system
\eqref{eq:partitioned_descriptor_system} but not of the nonsingular
system \eqref{eq:associated_nonsingular_system}. Let $x =
[x_{1}^{\transpose} \; x_{2}^{\transpose}]^{\transpose} \neq 0$ and
$u$ be the associated state-zero and input-zero direction,
respectively. Since $\Ker(\subscr{P}{singular}(s)) =
\Ker(\subscr{P}{nonsingular}(s))$ and $s$ is not a zero of the
nonsingular system \eqref{eq:associated_nonsingular_system}, it
follows that $x_{1} = 0$ and $x_{2} \neq 0$. Accordingly, we have that
\begin{equation*}
\Ker\left(
\begin{bmatrix}
-A_{12} & -B_{1} \\ -A_{22} & -B_{2} \\ C_{2} & D
\end{bmatrix}
\right)
\neq \{\emptyset\}
\,.
\end{equation*}
It follows that the vector $[0^{\transpose} \; x_{2}^{\transpose}
\; u^{\transpose}]^{\transpose}$ lies in the nullspace of
$\subscr{P}{singular}(s)$ for each $s \in \mathbb C$, and thus the
descriptor system \eqref{eq:partitioned_descriptor_system} is not
left-invertible. In conclusion, if the descriptor system
\eqref{eq:partitioned_descriptor_system} is left-invertible, then its
invariant zeros coincide with those of the nonsingular system
\eqref{eq:associated_nonsingular_system}.
\end{proof}

It should be noticed that, because of Theorem \ref{thm:inv_zero},
under the assumption of left-invertibility, classical linear systems
results can be used to investigate the presence of structural
undetectable attacks in a cyber-physical system; see
\cite{JMD-CC-JW:03} for a survey of results on generic properties
of linear systems.


\section{Illustrative examples}\label{sec:example}
\subsection{An example of state attack against a power network}
Consider the power network model analyzed in Example \ref{Example:
  power network structural analysis} and illustrated in Fig.
\ref{power_network}, and let the variables $\theta_4$ and $\theta_5$
be affected, respectively, by the unknown and unmeasurable signals
$u_1(t)$ and $u_2(t)$. Suppose that a monitoring unit is allowed to
measure directly the state variables of the first generator, that is,
$y_1(t)=\delta_1(t)$ and $y_2(t)=\omega_1(t)$.


\begin{figure}
  \centering
  \includegraphics[width=.65\columnwidth]{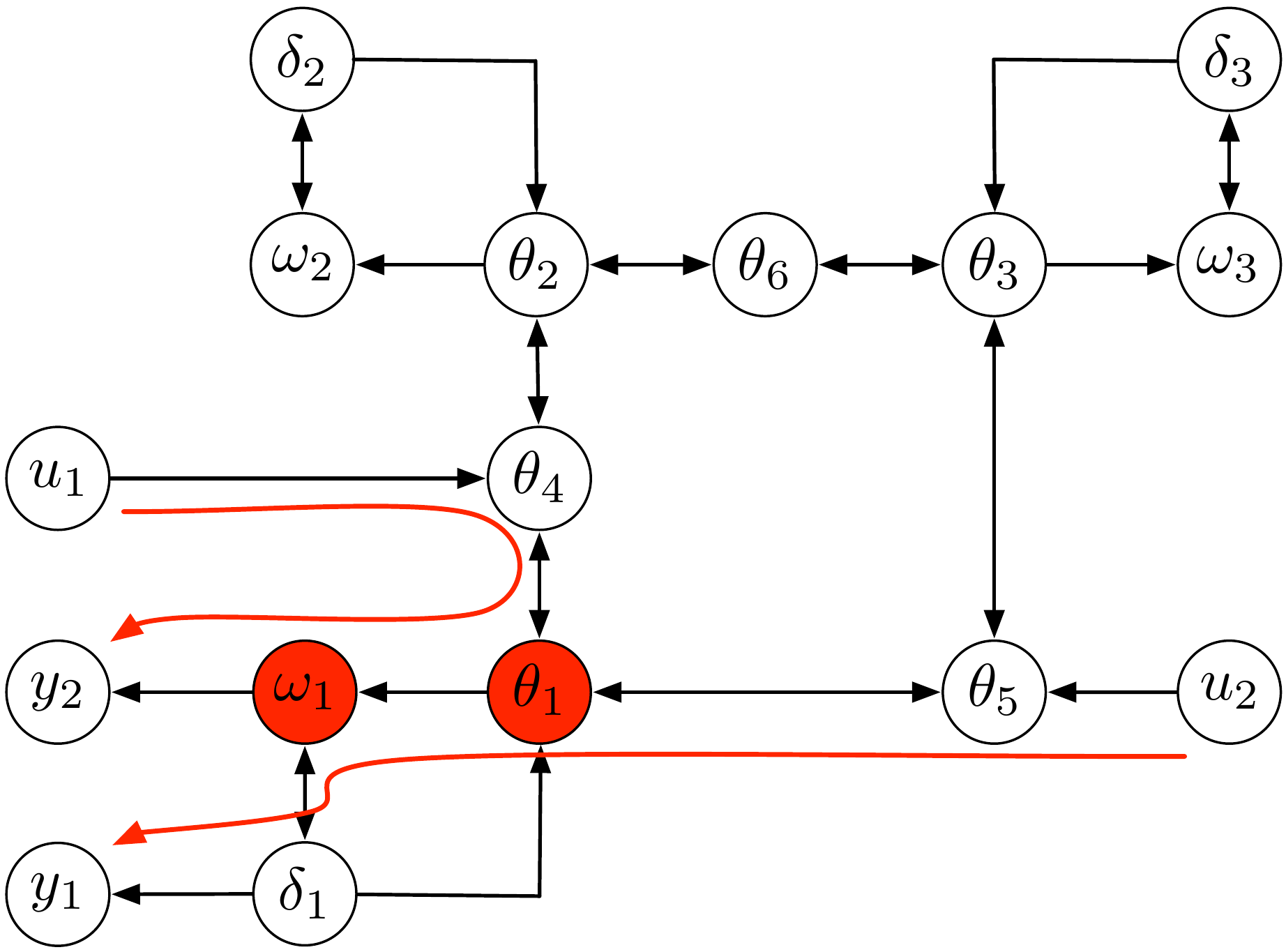}\\
  \caption{In the above network, there is no linking of size $2$ from
    the input to the output vertices. Indeed, the vertices $\theta_1$
    and $\omega_1$ belong to every path from $\{u_1,u_2\}$ to
    $\{y_1,y_2\}$. Two input to output paths are depicted in
    red.}\label{digraph_paths}
\end{figure}
Notice from Fig. \ref{digraph_paths} that the maximum size of a
linking from the failure to the output vertices is $1$, so that, by
Theorem \ref{thm:vulnerability}, there exists a structural
vulnerability. In other words, for every choice of the network
matrices, there exist nonzero $u_1(t)$ and $u_2(t)$ that are not
%
%
%
%
detectable through the measurements.\footnote{When these ouput-nulling
  inputs $u_1(t)$, $u_2(t)$ are regarded as additional loads, then
  they are entirely sustained by the second and third generator.}
  
We now consider a numerical realization of this system. Let the input
matrices be $B=[e_8 \; e_9]$ and $D = [0 \; 0]^{\transpose}$, the
measurement matrix be $C = [e_1 \; e_4]^\transpose$, and the system
matrix $A$ be as in equation \eqref{eq: power network descriptor
  system model} with $M_{g} = \text{blkdiag}(.125,.034,.016)$, $D_{g}
= \text{blkdiag}(.125,.068,.048)$, and
\begin{align*}
  \mc L = \!\left[
    \begin{smallmatrix}
      .058 & 0 & 0 & -.058 & 0 & 0 & 0 & 0 & 0\\
         0 & .063 & 0 & 0 & -.063 & 0 & 0 & 0 & 0\\
         0 & 0 & .059 & 0 & 0 & -.059 & 0 & 0 & 0\\
         -.058 & 0 & 0 & .235 & 0 & 0 & -.085 & -.092 & 0\\
         0 & -.063 & 0 & 0 & .296 & 0 & -.161 & 0 & -.072\\
         0 & 0 & -.059 & 0 & 0 & .330 & 0 & -.170 & -.101\\
         0 & 0 & 0 & -.085 & -.161 & 0 & .246 & 0 & 0\\
         0 & 0 & 0 & -.092 & 0 & -.170 & 0 & .262 & 0\\
         0 & 0 & 0 & 0 & -.072 & -.101 & 0 & 0 & .173
    \end{smallmatrix}
  \right].
\end{align*}
Let $U_1(s)$ and $U_2(s)$ be the Laplace transform of the attack
signals $u_1(t)$ and $u_2(t)$, and let
\begin{align*}
  \begin{bmatrix}
    U_1(s) \\ U_2(s)
  \end{bmatrix}
  =
  \underbrace{
  \begin{bmatrix}
    \frac{-1.024 s^4 - 5.121 s^3 - 10.34 s^2 - 9.584 s - 3.531}{s^4 + 5 s^3 + 9.865 s^2 + 9.173 s + 3.531}\\
    1
  \end{bmatrix}}_{\mathcal{N}(s)}
    \bar U(s),
\end{align*}
for {\em some arbitrary} nonzero signal $\bar U(s)$. Then it can be
verified that the failure cannot be detected through the measurements
$y_1(t)$ and $y_2(t)$. In fact, $\mathcal{N}(s)$ coincides with the
null space of the input/output transfer matrix.
An example is in Fig. \ref{sim_unstable}, where the second and the
third generator are driven unstable by the attack, but yet the first
generator does not deviate from the nominal operating condition.
\begin{figure}
  \centering
  \includegraphics[width=.65\columnwidth]{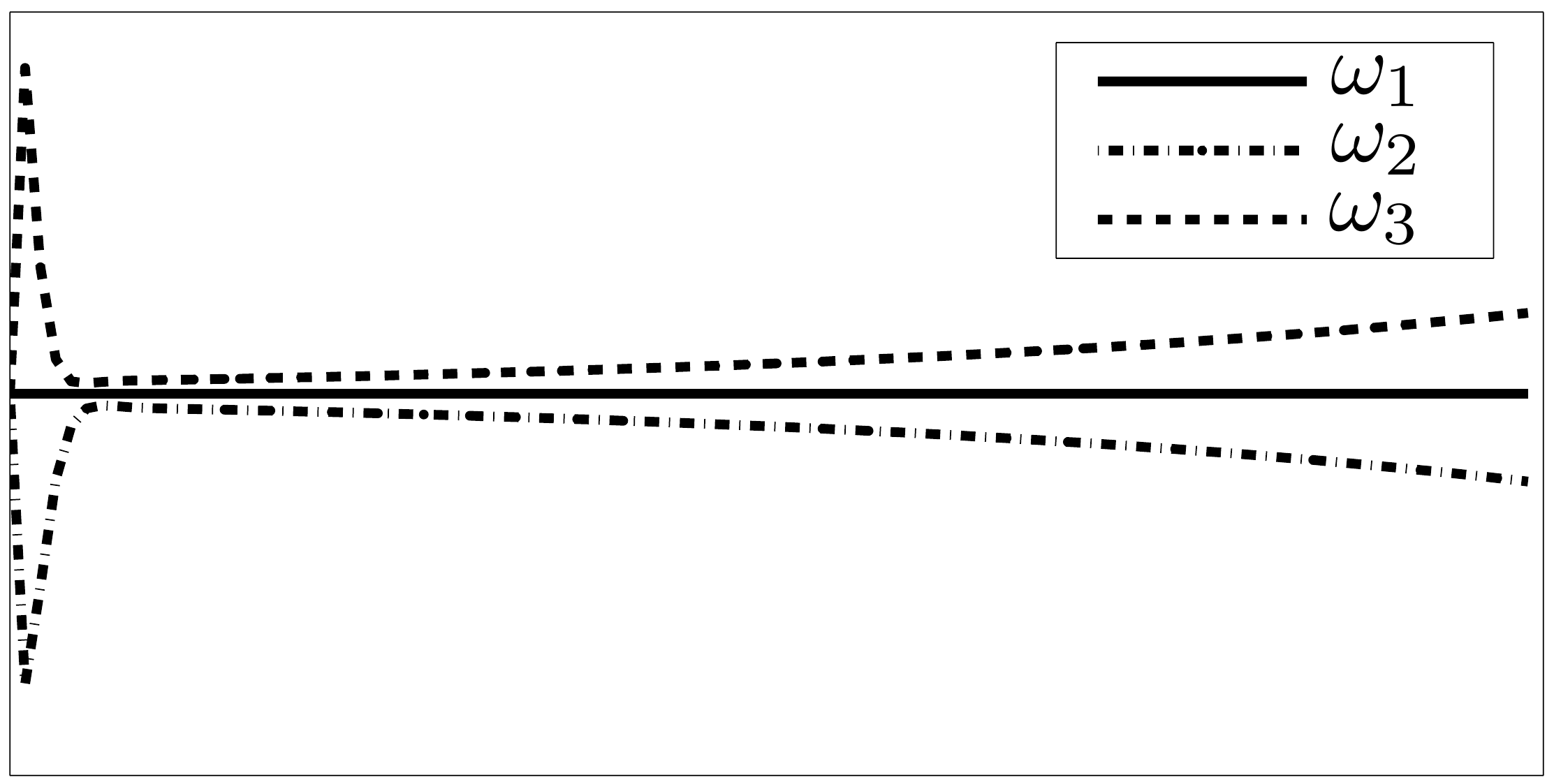}\\
  \caption{The velocities $\omega_2$ and $\omega_3$ are driven
    unstable by the signals $u_1(t)$ and $u_2(t)$, which are
    undetectable from the measurements of $\omega_1$ and
    $\delta_1$.}\label{sim_unstable}
\end{figure}

Suppose now that the rotor angle of the first generator and the
voltage angle at the $6$-th bus are measured, that is, $C=[e_1 \;
e_{12}]^\transpose$.  Then, there exists a linking of size $2$ from
$\mathcal{U}$ to $\mathcal{Y}$, and the system $(E,A,B,C)$ is
left-invertible. Following Theorem \ref{thm:inv_zero}, the invariant
zeros of the power network can be computed by looking at its reduced
system, and they are $-1.6864 \pm 1.8070i$ and $-0.8136 \pm
0.2258i$. Consequently, if the network state is unknown at the failure
time, there exists vulnerabilities that an attacker may exploit to
affect the network while remaining undetected. Finally, we remark that
such state attacks are entirely realizable by cyber attacks
\cite{AHMR-ALG:11}.

\subsection{An example of output attack against a power
  network}\label{sec:example_2}

\begin{figure}
    \centering
    \includegraphics[width=.9\columnwidth]{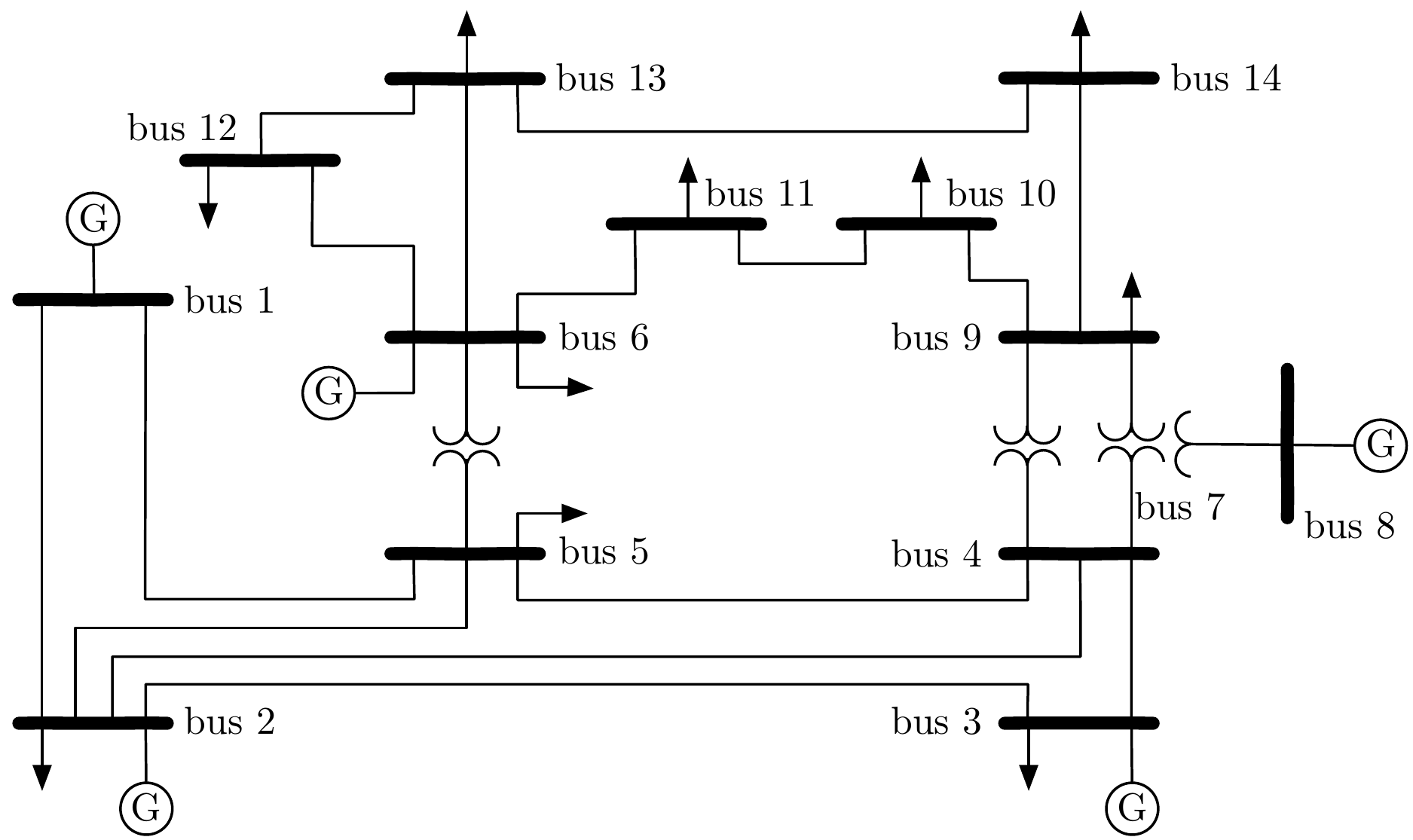}
    \caption{For the here represented IEEE 14 bus system, if the
      voltage angle of one bus is measured exactly, then a cyber
      attack against the measurements data is always detectable by our
      dynamic detection procedure. In contrary, as shown in
      \cite{YL-MKR-PN:09}, a cyber attack may remain undetected by a
      static procedure if it compromises as few as four measurements.}
    \label{fig:ieee14}
\end{figure}
Let the IEEE 14 bus power network (Fig. \ref{fig:ieee14}) be modeled
as a descriptor system as in Section \ref{example:power}. Following
\cite{YL-MKR-PN:09}, let the measurement matrix $C$ consist of the
real power injections at all buses, of the real power flows of all
branches, and of one rotor angle (or one bus angle). We assume that an
attacker can compromise all the measurements, independently of each
other, except for one referring to the rotor angle.

Let $k \in \mathbb{N}_0$ be the cardinality of the attack set. It is
known that an attack undetectable to a static detector exists if $k
\ge 4$ \cite{YL-MKR-PN:09}. In other words, due to the sparsity
pattern of $C$, there exists a signal $u_K(t)$, with (the same) four
nonzero entries at all times, such that $D u_K(t) \in \Image (C)$ at
all times. By Theorem \ref{Theorem: Static detectability of
  cyber-physical attacks} the attack set $K$ remains undetected by a
Static Detector through the attack mode $u_K (t)$. On the other hand,
following Theorem \ref{Theorem: Dynamic detectability of
  cyber-physical attacks}, it can be verified that, for the same
output matrix $C$, and independent of the value of $k$, there exists
\emph{no} undetectable (output) attacks for a dynamic monitor.

It should be notice that this result relies on the fact that the rotor
angle measurement is known to be correct, because, for instance, it is
protected using sophisticated and costly security methods
\cite{ARM-RLE:10}. Since the state of the IEEE 14 bus system can be
reconstructed by means of this measurement only (in a system theoretic
sense, the system is observable by measuring one generator rotor
angle), the output attack $D u(t)$ is easily identified as $D u(t) =
y(t) - C \hat x(t)$, where $\hat x (t) = x(t)$ is the reconstructed
system state at time $t$.

\subsection{An example of state and output attack against a water
  supply network}
\begin{figure}
    \centering
    \includegraphics[width=1\columnwidth]{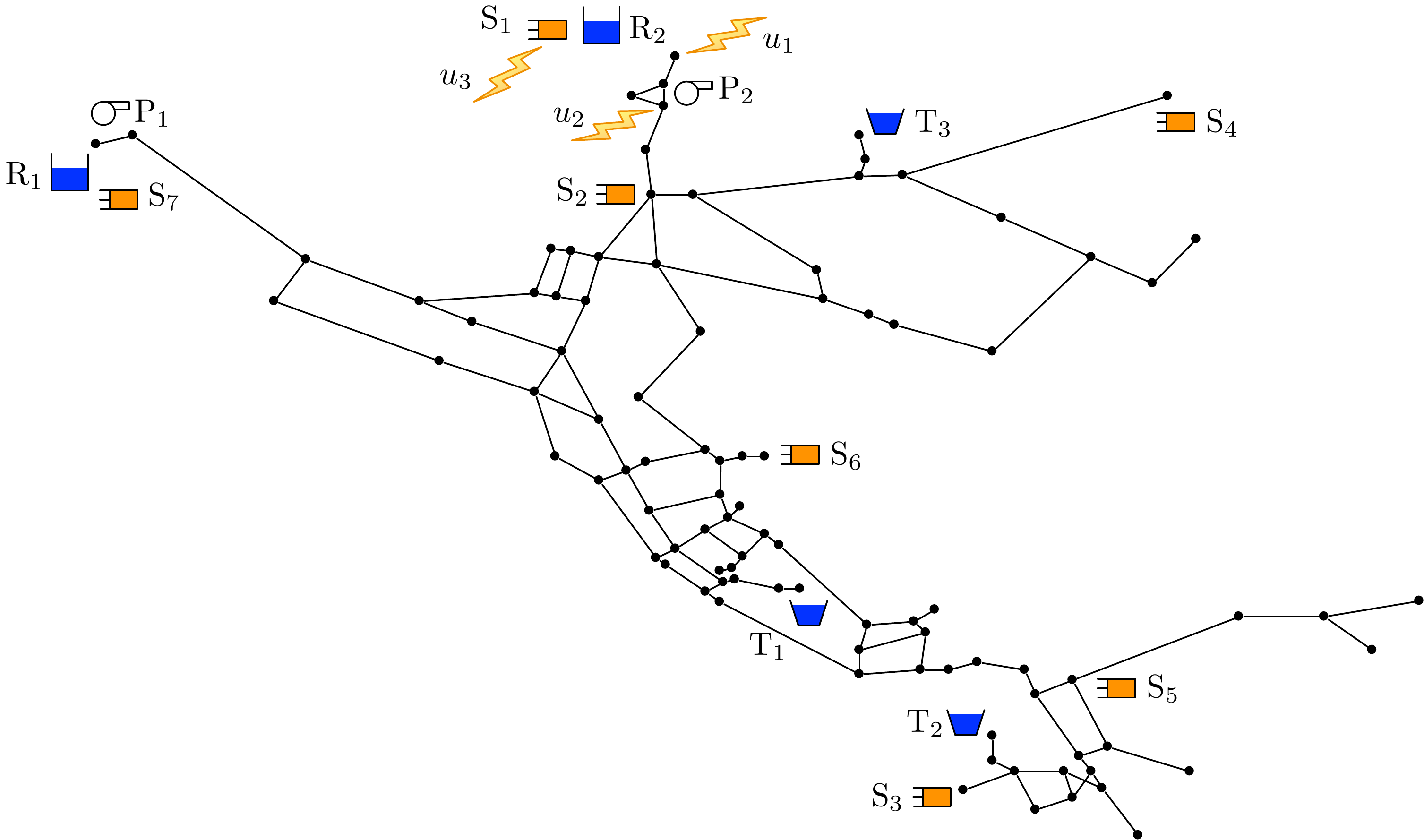}
    \caption{This figure shows the structure of the EPANET water
      supply network model \# 3, which features $3$ tanks
      ($\text{T}_1$, $\text{T}_2$, $\text{T}_3$), $2$ reservoirs
      ($\text{R}_1$, $\text{R}_2$), $2$ pumps ($\text{P}_1$,
      $\text{P}_2$), $96$ junctions, and $119$ pipes. Seven pressure
      sensors ($\text{S}_1, \dots, \text{S}_7$) have been installed to
      monitor the network functionalities. A cyber-physical attack to
      steal water from the reservoir $\text{R}_2$ is reported. Notice
      that the cyber-physical attack features two state attacks
      ($u_1$, $u_2$) and one output attack ($u_3$).}
    \label{fig:epanet3}
\end{figure}
 
Consider the water supply network EPANET 3 linearized at a steady
state with non-zero pressure drops \cite{EPANET2:00}. The water
network model as well as a possible cyber-physical attack are
illustrated in Fig. \ref{fig:epanet3}. The considered cyber-physical
attack aims at stealing water from the reservoir $\text{R}_2$ while
remaining undetected from the installed pressure sensors $\text{S}_1,
\dots, \text{S}_7$. In order to achieve its goal, the attacker
corrupts the measurements of sensor $\text{S}_1$ (output attack), it
steals water from the reservoir $\text{R}_2$ (state attack), and,
finally, it modifies the input of the control pump $\text{P}_2$ to
restore the pressure drop due to the loss of water in $\text{R}_2$
(state attack). We now analyze this attack in more details.

Following the modeling in Section \ref{example:water}, an index-one
descriptor model describing the evolution of the water network in
Fig. \ref{fig:epanet3} is computed. For notational convenience, let
$x_1 (t)$, $x_2 (t)$, $x_3 (t)$, and $x_4 (t)$ denote, respectively,
the pressure at time $t$ at the reservoir $\text{R}_2$, at the
reservoir $\text{R}_1$ and at the tanks $\text{T}_1$, $\text{T}_2$ and
$\text{T}_3$, at the junction $\text{P}_2$, and at the remaining
junctions. The index-one descriptor model reads as
\begin{align*}
  \begin{bmatrix}
    \dot x_1(t) \\ M \dot x_2(t) \\ 0 \\ 0
  \end{bmatrix}
  &=
  \begin{bmatrix}
    0 & 0  & 0 & 0\\
    0 & A_{22} & 0 & A_{24}\\
    A_{31} & 0 & A_{33} & A_{34}\\
    0 & A_{42} & A_{43} & A_{44}
  \end{bmatrix}
  \begin{bmatrix}
    x_1(t) \\ x_2(t) \\ x_3(t) \\ x_4(t)
  \end{bmatrix},
\end{align*}
where the pattern of zeros is due to the network interconnection
structure, and $M = \text{diag}(1,A_1,A_2,A_3)$ corresponds to the
dynamics of the reservoir $\text{R}_1$ and the tanks $\text{T}_1$,
$\text{T}_2$, and $\text{T}_3$. With the same partitioning, the attack
signature reads as $B = [B_1 \; B_2 \; 0]$ and $D = [0 \; 0 \; D_1]$,
where
\begin{align*}
  B_1 &=
  \begin{bmatrix}
    1 & 0 & 0 & 0
  \end{bmatrix}^\transpose,\;
  B_2 =
  \begin{bmatrix}
    0 & 0 & 1  & 0
  \end{bmatrix}^\transpose,\; \text{and } \\
  D_1 &=
  \begin{bmatrix}
    1 & 0 & \dots & 0
  \end{bmatrix}^\transpose.
\end{align*}
Let the attack $u_2 (t)$ be chosen as $u_2(t) = - A_{31}
x_1(t)$. Then, the state variables $x_2$, $x_3$, and $x_4$ are
decoupled from $x_1$. Consequently, the attack mode $u_1$ does not
affect the dynamics of $x_2$, $x_3$, and $x_4$. Let $u_1(t) = -1$, and
notice that the pressure $x_1 (t)$ decreases with time (that is, water
is being removed from $\text{R}_2$). Finally, for the attack to be
undetectable, since the state variable $x_1$ is continuously monitored
by $\text{S}_1$, let $u_3 (t) = - x_1(t)$. It can be verified that the
proposed attack strategy allows an attacker to steal water from the
reservoir $\text{R}_2$ while remaining undetected from the sensors
measurements. In other words, the attack $(Bu(t),Du(t))$, with $u(t) =
[u_1^\transpose (t) \; u_2^\transpose (t) \; u_3^\transpose
(t)]^\transpose$, excites only zero dynamics for the water network
system in Fig. \ref{fig:epanet3}.

We conclude this section with the following remarks. First, for the
implementation of the proposed attack strategy, neither the network
initial state, nor the network structure besides $A_{31}$ need to be
known to the attacker. Second, the effectiveness of the proposed
attack strategy is independent of the sensors measuring the variables
$x_3$ and $x_4$. On the other hand, if additional sensors are used to
measure the flow between the reservoir $\text{R}_2$ and the pump
$\text{P}_2$, then an attacker would need to corrupt these
measurements as well to remain undetected. Third and finally, due to
the reliance on networks to control actuators in cyber-physical
systems, the attack $u_2 (t)$ on the pump $\text{P}_2$ could be
generated by a cyber attack \cite{AHMR-ALG:11}.


\section{Conclusion}\label{sec:conclusion}
For cyber-physical systems modeled by linear time-invariant descriptor
systems, we have analyzed fundamental limitations of static, dynamic,
and active attack detection and identification monitors. We have
rigorously shown that a dynamic detection and identification monitor
exploits the network dynamics and outperforms the static counterpart,
while requiring, possibly, fewer measurements. Additionally, we have
shown that active monitors have the same limitations as passive
dynamic monitors. Finally, we have described graph theoretic
conditions for the existence of undetectable and unidentifiable
attacks. These latter conditions exploit the system interconnection
structure, and they hold for almost all compatible numerical
realizations. In the companion paper \cite{FP-FD-FB:10yb} we develop
centralized and distributed attack detection and identification
monitors.


\bibliographystyle{IEEEtran}
\bibliography{alias,Main,FB}









\end{document}